\documentclass[psamsfonts,leqno,12pt]{amsart}
\usepackage [dvips]{graphicx}
\usepackage{epsf}
\usepackage{amssymb,amsmath,amssymb,xy,verbatim,amscd}
\usepackage{amsfonts}

\theoremstyle{plain}
\newtheorem{theorem}{Theorem}[section]
\newtheorem{proposition}[theorem]{Proposition}
\newtheorem{lemma}[theorem]{Lemma}

\theoremstyle{definition}
\newtheorem{definition}[theorem]{Definition}
\newtheorem{corollary}[theorem]{Corollary}
\newtheorem{example}[theorem]{Example}
\newtheorem{remark}[theorem]{Remark}

\numberwithin{equation}{theorem}

\newcommand{\C}{{\mathbb C}}
\newcommand{\R}{{\mathbb R}}
\newcommand{\Z}{{\mathbb Z}}

\newcommand{\Q}{{\mathbb Q}}

\renewcommand{\ll}{{\langle}}
\newcommand{\rr}{{\rangle}}
\newcommand{\reg}{{\mathrm{reg}}}
\newcommand{\CR}{{\mathcal{R}}}
\newcommand{\A}{{\mathcal{A}}}

\newcommand{\CH}{{\mathcal{H}}}
\newcommand{\CT}{{\mathcal T}}

\newcommand{\Eis}{\operatorname{Eis}}

\newcommand{\s}{\mathfrak{s}}
\renewcommand{\a}{\mathfrak{a}}
\newcommand{\Res}{\operatorname{Res}}

\newcommand{\res}{\operatorname{res}}
\newcommand{\Ber}{\operatorname{Ber}}

\renewcommand{\ll}{{\langle}}

\newcommand{\aff}{{\mathrm{aff}}}

\begin{document}
\title{Multiple Bernoulli series, an Euler-MacLaurin formula, and Wall crossings}
\author{Arzu Boysal and Mich\`{e}le Vergne}
\date{}
\maketitle

\pagestyle{myheadings}

\markboth{Arzu Boysal and Mich\`{e}le Vergne}{Multiple Bernoulli series}

\begin{abstract}
Using multiple Bernoulli series, we give a formula in the spirit of
Euler-MacLaurin formula. We also give a wall crossing formula and a
decomposition formula for multiple Bernoulli series. The study of
these  series is motivated by  formulae of E.~Witten for volumes of
moduli spaces of flat bundles over a surface.
\end{abstract}

{\small \tableofcontents}

\section{Introduction}

Consider a sequence of vectors $\Phi$ lying in a lattice  $\Lambda$
of a real vector space $V$. We denote the dual lattice of $\Lambda$ by $\Gamma$.  Let
$\Gamma_{\reg}(\Phi)=\{\gamma\in \Gamma |\; \ll \phi,\gamma \rr \neq
0,\;\mbox{for all} \;\phi\in \Phi\}$ be the set of regular elements in $\Gamma$. Let $Z$ be the fundamental
domain in $V$ for $\Lambda$ and $dv$ the  Lebesgue measure giving
measure $1$ to $Z$. Here in introduction we freely identify distributions
and generalized functions via this choice of $dv$.
\medskip

In this paper we study the distribution $\mathcal{B}(\Phi,\Lambda)$
on the torus $V/\Lambda$ defined via its Fourier coefficients as:
\[\int_ Z \mathcal{B}(\Phi, \Lambda)(v) e^{-\ll 2i\pi v , \gamma \rr}dv =
\begin{cases} \frac{1}{\prod_{\phi \in \Phi}\ll 2i\pi \phi,\gamma \rr} & \text{if $\gamma \in \Gamma_{\reg}(\Phi)$,}\\
0 &\text{otherwise.}
\end{cases}
\]
We have
\[\mathcal{B}(\Phi, \Lambda)(v)=\sum_{\gamma \in \Gamma_{\reg}(\Phi)}  \frac{e^{\ll 2i\pi v ,\gamma \rr}}{
\prod_{\phi \in \Phi} \ll 2i\pi \phi,\gamma \rr}. \] This sum, if
not absolutely convergent, has meaning as a distribution.

\medskip

We call $\mathcal{B}(\Phi, \Lambda)$ the \textit{multiple Bernoulli series}
associated to $\Phi$ and $\Lambda$.  They are natural  generalizations of Bernoulli series:
for $\Lambda=\Z \omega$ and $\Phi_k=[\omega,\omega,\ldots, \omega]$,
where  $\omega$ is repeated $k$ times with $k>0$, the distribution
$$\mathcal{B}(\Phi_k,\Lambda)(t\omega)=\sum_{n\neq 0} \frac{e^{2i\pi nt}}{
(2i\pi n)^k}$$ is equal to $-\frac{1}{k!} B(k,t-[t])$ where
$B(k,t)$ denotes the $k^{\text{th}}$ Bernoulli polynomial in
variable $t$.

\medskip

Multiple Bernoulli  series appeared in the work of E.~Witten in the special case
where the sequence $\Phi$ is comprised of positive coroots of a
compact connected Lie group $G$ with multiplicity $2g-1$ and
$\Lambda$ is the coroot lattice of $G$.  Witten shows that
(\cite{wi}, \S 3)  for the above instance of $\Phi$ and $\Lambda$
and for a regular element $v \in Z$, the value of $\mathcal{B}(\Phi,
\Lambda)(v)$ is (upto a scalar depending on $G$ and $g$) the
symplectic volume of the moduli space $\mathcal{M}(G,g,v) $ of flat $G$-connections on
Riemann surface of genus $g$ with one boundary component, around
which the holonomy is determined by $v$.

For example, consider $G=\text{SU}(3)$, denote its simple roots by $\{\alpha_1,\alpha_2\}$ and associated coroots by $\{H_{\alpha_1},H_{\alpha_2}\}$.
Then on a  Riemann surface of genus $g=2$,
the symplectic volume of the moduli space of flat $\text{SU}(3)$-connections with one boundary component  marked by
$v=a_1H_{\alpha_1}+a_2H_{\alpha_2}$ (lying in the fundamental alcove) is given by the following sum
$$\sum_{\substack{n_1\in \Z,\, n_2\in\Z\\ n_1\neq 0,\, n_2\neq 0,\, n_1+n_2\neq 0}}
\frac{e^{2i\pi(n_1a_1+n_2a_2)}}{(2i\pi n_1)^3(2i\pi n_2)^3(2i\pi (n_1+n_2))^3} $$
up to a scalar multiple.

These series have been  extensively studied by A. Szenes (\cite{sze1},\cite{sze2}),
who gave multidimensional residue formulae for them.

\medskip

If $f$ is a function on the real line, smooth and sufficiently decreasing, also with sufficiently decreasing derivatives,
then the Euler-MacLaurin formula  gives
$$\sum_{n\in \Z} f(n)=\int_{\R} f(t) dt+(-1)^{k-1}\frac{1}{k!}\int_{\R}  B(k,t-[t]) f^{[k]}(t)dt$$
where $f^{[k]}$ denotes the $k^{\text{th}}$ derivative of $f$.

We give a natural generalization of this formula involving $\mathcal{B}(\Phi, \Lambda)$ in Theorem \ref{eumacl}.
The difference between the discrete sum $\sum_{\lambda\in \Lambda}f(\lambda)$ and the integral of $f$ over $V$ will only involve  derivatives
$(\prod_{\phi\in Y} \partial_ \phi) f$  over `long subsets' $Y$ of $\Phi$, that is, subsets $Y$ such that  their complement
in $\Phi$ do not span the vector space $V$.

\medskip

We start with giving some properties of the distribution $\mathcal{B}(\Phi,\Lambda)$
that are pertinent for what follows.

\medskip

The distribution $\mathcal{B}(\Phi,\Lambda)$ is periodic with
respect to $\Lambda$.
Moreover, it satisfies a certain recurrence relation which we
outline next.

\medskip
We will call a set with multiplicities \textit{a list}.
Suppose $\Phi$ contains $\phi$ with multiplicity $m$,
then we denote the list that contains $\phi$ with multiplicity $m-1$ by $\Phi-\{\phi\}$; whereas we denote the
list where all copies of $\phi$ are removed by $\Phi\setminus\{\phi\}$.  More generally, for a subset $B$ of $V$, by  $\Phi\setminus B$ we mean the
list of elements of $\Phi$ not lying in $B$.
By   $\Phi\cap B$ we mean the
list of elements of $\Phi$  lying in $B$.

For an element $\phi$ in $\Phi$, we associate two lists
of vectors as follows: First, we consider the list $\Phi-
\{\phi\}$ in $V$, and respectively the distribution
$\mathcal{B}(\Phi- \{\phi\}, \Lambda)$ on $V$. Let
$V_0:=V/\R \phi$ and let $\Lambda_0$ denote the image  of the lattice
$\Lambda$ in $V_0$. Secondly, we consider the list $\Phi_0$  of
elements of $V_0$ consisting of the images  of the elements in
$\Phi-\{\phi\}$.  Then we may consider
$\mathcal{B}(\Phi_0,\Lambda_0)$ as a distribution on $V$ `constant
in the direction of $\phi$'.  Observe that if $\Phi$ contains more than one copy of $\phi$ then
$\Phi_0$ contains the zero vector and consequently $\mathcal{B}(\Phi_0,\Lambda_0)$ is identically zero.

It is clear that the
distribution $\mathcal{B}(\Phi,\Lambda)$ satisfies the following
recurrence relation
\begin{equation}\label{introrec}
\partial_\phi \mathcal{B}(\Phi,\Lambda)=\mathcal{B}(\Phi-\{\phi\},\Lambda)-\mathcal{B}(\Phi_0,\Lambda_0).
\end{equation}

Assuming that  $\Phi=[\phi_1,\phi_2,\ldots,\phi_N]$ spans a pointed
cone, we may define   the tempered distribution $T(\Phi)$ defined on
test functions $f$ by:
\begin{equation}\label{multispline1}
\langle T(\Phi)\,|\,f\rangle = \int_0^\infty\cdots\int_0^\infty
f(\sum_{i=1}^Nt_i \phi_i)dt_1\cdots dt_N.
\end{equation}
In other words, $T([\phi])$ is the Heaviside distribution $\ll H(\phi),f\rr=\int_0^{\infty}f(t\phi) dt$
and $T(\Phi)$ is the convolution product
 $$H(\phi_1)*H(\phi_2)*\cdots *H(\phi_N)$$ of the Heaviside distributions $H(\phi_k)$.

If $\Phi$ generates $V$, $T(\Phi)$  is a positive measure on $V$
given by  integration against a piecewise polynomial function called a
\textit{multispline}. For any $\phi\in \Phi$,
\begin{equation}\label{introdm} \partial_\phi T(\Phi)=T(\Phi-\{\phi\}).\end{equation}
We remark the similarity between the recurrence
relations (\ref{introrec}) and (\ref{introdm}).
In fact we will express $\mathcal{B}(\Phi,\Lambda)$ in terms of superposition of multispline functions in
Theorem \ref{decomposition}.
\medskip

If $\Phi$ generates $V$, then  the periodic function $\mathcal{B}(\Phi,\Lambda)$ is
\textit{piecewise polynomial}; this we reprove in Section \ref{section:poly}.

\bigskip

In the rest of the introduction, for simplicity, we assume that $\Phi$ generates $V$.
We call a connected component of the complement of affine walls (that is, hyperplanes that are generated by some elements of
$\Phi$ and their
translates with respect to the lattice $\Lambda$) in $V$ a \textit{tope}.
For example, Figure \ref{a2} depicts topes associated to the system $\Phi=[e_1,e_2,e_1+e_2]$ and $\Lambda=\Z e_1 \oplus \Z e_2$.

\begin{figure}
\begin{center}
 \includegraphics[width=37mm]{tope2mod.mps}\\
 \caption{$\CT([e_1,e_2,e_1+e_2],\Z e_1 \oplus \Z e_2)$}\label{a2}
 \end{center}
\end{figure}

Given a tope $\tau$ associated to the system $(\Phi,\Lambda)$, we
denote by $\Ber(\Phi,\Lambda,\tau)$ the \textit{polynomial} function
on $V$ such that the restriction of $\mathcal{B}(\Phi,\Lambda)$ to
$\tau$ coincides with the restriction of
$\Ber(\Phi,\Lambda,\tau)(v)$  to $\tau$.

\medskip

Let $W$ be an hyperplane in $V$
spanned by some elements of $\Phi$, and let $E \in \Gamma$ be an
equation of this hyperplane.
We reverse the directions of `half of'  the $\phi_i$ in $\Phi\setminus W$ in order that they all lie in the strict half space determined by $E$,
and define
$$T(\Phi\setminus W,E):=\displaystyle
\prod_{\langle \phi_i,E\rangle <0}-H(-\phi_i)*\prod_{\langle \phi_i,E\rangle > 0} H(\phi_i).$$

$T(\Phi\setminus W,E)$ is a distribution supported on $E\geq 0$. We similarly define
$T(\Phi\setminus W,-E)$.

Now we compare
the polynomials $\Ber(\Phi,\Lambda,\tau)$ associated to two adjacent
topes separated by an hyperplane $W$ (cf. Section \ref{section:Jump}). The jump can be expressed in terms of a lower dimensional multiple Bernoulli
series and a multispline function.
More explicitly, we have the following wall crossing formula:
\begin{theorem}\label{introthe}
Let $\tau_1$ and $\tau_2$ be two adjacent topes separated by the
hyperplane $W$ defined by $E$ with $\langle v,E\rangle >0$ for any $v \in \tau_1$. Denote by
$\tau_{12}$ the tope with respect to the system $(\Phi \cap W,
\Lambda\cap W)$ containing $\overline{\tau_1} \cap
\overline{\tau_2}$ in its closure. Let
$\Ber^{\tau_{12}}:=\Ber(\Phi\cap W,\Lambda \cap W,\tau_{12})dh$ be
the polynomial  density  on $W$ determined by $\tau_{12}$. Then,
$$(\Ber(\Phi,\Lambda,\tau_1)-\Ber(\Phi,\Lambda,\tau_2))dv=
\Ber^{\tau_{12}} *T(\Phi\setminus W,E)-
\Ber^{\tau_{12}} *T(\Phi\setminus W,-E).$$
\end{theorem}

The left hand side of the above equation is a polynomial density;
it is easily proven that the right hand side is also a polynomial density.

The wall crossing formula given in the above theorem is analogous
to the formula in Boysal-Vergne \cite{bover}. This formula is
also similar to wall crossing formulae in Hamiltonian geometry for
the push-forward of the Liouville measure; indeed, when crossing a
wall, this piecewise polynomial measure changes according to the same
scheme \cite{par}, \cite{gu}.  Our wall crossing formula in Theorem \ref{introthe} is thus in accordance with the fact that for
special cases $\mathcal{B}(\Phi,\Lambda)$ computes the volume of the moduli spaces
${\mathcal M}(G,g,v)$.  These spaces can be described as symplectic reduction at $v$ of the Jeffrey-Kirwan extended moduli space $M(G,g)$, so that
their volume at $v$ is given by  the push-forward of the Liouville measure on $M(G,g)$, a piecewise polynomial function periodic with respect to a
lattice $\Lambda$.
Recall that Jeffrey-Kirwan (\cite{jef-kir}) proved  wall crossing formulae  for   integrals on moduli spaces $\mathcal{M}(G,g,v)$,
and used them in a fundamental way to compute intersection pairings on  $\mathcal{M}(G,g,v)$  when $G=SU(n)$.
However, in the general situation that we are considering here, we do not have such a geometric
interpretation of the multiple Bernoulli series.

\medskip

In Section \ref{section:affine} we generalize the above results to the case of affine arrangements.

\medskip

In Section \ref{section:decomp}, we  give a decomposition formula for
$\mathcal{B}(\Phi,\Lambda)$, describing it as a superposition of `basic
pieces'  made of convolution products of lower dimensional Bernoulli series and multisplines.
More precisely,  we say that $\s$ is an admissible subspace of $V$ if $\s$ is spanned by some elements of $\Phi$, and we say that $\a$ is affine
admissible, if $\a$ is a translate by $\Lambda$ of an admissible subspace.
Given a tope $\tau$, we express the difference between the {\bf piecewise polynomial} density $\mathcal{B}(\Phi,\Lambda)$ and the {\bf polynomial}
density $\Ber(\Phi, \Lambda,\tau)$ as a sum of distributions $\mathcal A(\Phi,\Lambda,\a,\beta)$ associated to proper affine admissible subspaces
$\a$ and the choice of an element $\beta\in \tau$.
The supports of these distributions  do not intersect $\tau$ and are convolution products of polynomial distributions supported on $\a$ with
multisplines distributions directed towards the exterior of $\tau$.  Our construction is inspired by the stratification of a Hamiltonian manifold $M$
using the square of the moment map as Morse function, and we will use a scalar product on $V$.
Our decomposition formula is  very similar to Paradan's decomposition of the equivariant index  of a twisted Dirac operator on  $M$ \cite{par3}.
In \cite{sze-ve2}, Paradan's decomposition was proved by combinatorial methods, and used to give a proof that quantization commutes with
reduction for compact Hamiltonian manifolds. We follow here very closely the line of approach of \cite{sze-ve2}.  However our superposition  is an
infinite (but locally finite) superposition. This is in accordance with the fact that for some special cases,   our distributions are related to
Liouville measures of noncompact Hamiltonian manifolds such as $M(G,g)$ with infinite number of critical components for the square of the moment map.

For example, the periodic polynomial $-B(2,t-[t])/2$ in Figure~\ref{spline}(a) is decomposed in Figure~\ref{spline}(b) as a superposition of a polynomial
density with an infinite number of spline functions.

\begin{figure}
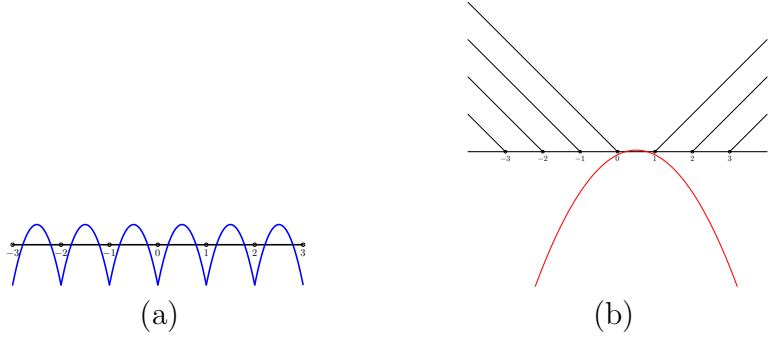

\begin{center}
\includegraphics[width=40mm]{spline2.mps}\hspace{2cm}
 \includegraphics[width=40mm]{spline1.mps}\\

  (a)\hspace{5.5cm}(b)
 \caption{The decomposition of $\mathcal{B}(\Phi_2,\Lambda)(t\omega)$.}\label{spline}
 \end{center}
\end{figure}

\subsection*{Acknowledgements} We thank Michel Duflo for various suggestions.
The first author wishes to thank \'{E}cole Polytechnique (Palaiseau, France), and the second author wishes to thank Bo\={g}azi\c{c}i \"{U}niversitesi
(Turkey) for financial support of research visits. The first author is partially supported by Bo\={g}azi\c{c}i \"{U}niversitesi (B.U.)
Research Fund $5076.$

\section*{List of Notations}
\[\begin{array}{ll}
V_{\mathbb Q} &\mbox{r-dimensional vector space over ${\mathbb Q}$.}\\
V &\mbox{the r-dimensional real vector space $V_{\mathbb Q}\bigotimes_{\mathbb Q}{\mathbb R}$; $v \in V$.}\\
U &\mbox{the dual of $V$; $x\in U$.}\\
\ll \;,\; \rr &\mbox{the pairing between $U$ and $V$.}\\
\Gamma &\mbox{a lattice in $U$; $\gamma\in \Gamma$.}\\
\Lambda:=\Gamma^* &\mbox{dual lattice in $V$; $\ll \Gamma,\Lambda\rr\subset \Z$, $\lambda\in \Lambda$.}\\
\Phi &\mbox{a sequence of  vectors in $V_\Q$; $\phi\in \Phi$.}\\
\mathcal{B}(\Phi, \Lambda) &\mbox{multiple Bernoulli series associated to $\Phi$ and $\Lambda$ (equation \ref{mbseriess}).}\\
B(k,t) &\mbox{$k^{th}$ Bernoulli polynomial.}\\
\mathcal{T}(\Phi, \Lambda) &\mbox{the set of topes associated to the system $(\Phi,\Lambda)$; $\tau$  a tope.}\\
T(X) &\mbox{multivariate spline distribution defined for a set of vectors $X$ in $V$.}\\
\mathcal{R} &\mbox{set of $\Phi$-admissible subspaces of $V$.}\\
W &\mbox{$\Phi$-admissible hyperplane.}\\
H_{\phi} &\mbox{hyperplane in $U$ comprising of vectors $u$ satisfying $\ll u,\phi \rr=0$.}\\
\mathcal{H}=\mathcal{H}(\Phi) &\mbox{hyperplane arrangement associated to $\Phi$.}\\
\mathcal{R}_{\mathcal{H}} &\mbox{ring of rational functions on $U$ with poles along $\mathcal{H}$.}\\
\mathcal{G}_{\mathcal{H}} &\mbox{a subspace of $\mathcal{R}_{\mathcal{H}}$ defined in  \ref{defgphi}.}\\
\end{array}\]

\section{Multiple Bernoulli series and hyperplane arrangements}\label{sec:defber}

Let $V_{\mathbb Q}$ be an r-dimensional vector space over ${\mathbb Q}$, and
let $V$ be the real vector space $V_{\mathbb Q}\bigotimes_{\mathbb Q}{\mathbb R}$.  Let $U$
denote the dual vector space to $V$.  Let  $\Lambda$ be a lattice in $V$ contained in $V_{\mathbb Q}$ and $\Gamma\subset
U$ be the dual lattice to $\Lambda$ so that $\ll
\Gamma,\Lambda\rr\subset \Z$.  For a subset $S$ of $V$, we denote by $<S>$ the subspace of $V$
generated by $S$.

\medskip

Let $Z$ denote the fundamental domain in $V$ for $\Lambda$.
Let $d_{\Lambda}v$ be the Lebesgue measure on $V$ giving  measure $1$ to $Z$. Our main
object of study is certain piecewise polynomial densities on $V$. For
our purposes it will be convenient to use the language of distributions.  If  $f$ is a
locally $L^1$ function on $V$, or more generally a generalized function, then  $f(v)d_{\Lambda}v$ is a distribution on $V$.
We use the notation $f(v)d_{\Lambda}v$, although the value of $f$ at the point $v\in V$ has
usually no meaning.
\medskip

For $v_0\in V$, the  translation ${\rm t}(v_0)$ acts on
distributions on $V$. If $D=f(v) d_{\Lambda}v$, then ${\rm
t}(v_0)D=f(v+v_0) d_{\Lambda}v$. We identify a distribution $D$ on
$V$ periodic with respect to $\Lambda$ (that is ${\rm
t}(\lambda)D=D$ for any $\lambda\in \Lambda$) to a distribution $D$
on the torus $V/\Lambda$.

\medskip

We will say that a locally $L^1$ function $f$ is \textit{piecewise polynomial}, if there exists a decomposition of $V$ in a union of
polyhedral pieces $C_i$ such that  the restriction of $f$ to $C_i$
is given by a polynomial formula.  We then say that the distribution
$f(v)d_{\Lambda}v$ is piecewise polynomial.

If $v \in V$, we  denote by $\delta_v$ the $\delta$ distribution at $v$:
$\langle \delta_v,f\rangle =f(v)$.
The Poisson formula reads as the following equality of distributions

\begin{equation}\label{Poisson}
\sum_{\lambda\in \Lambda} \delta_\lambda=\sum_{\gamma\in \Gamma}e^{2i\pi\langle \gamma,v\rangle } d_{\Lambda}v.
\end{equation}

\medskip
We now introduce the main object of study of this article.

\medskip

Let $\Phi$ be a sequence of vectors in $V_{\mathbb Q}$. Let
$$U_{\reg}(\Phi)=\{u \in U |\; \ll \phi,u \rr \neq
0,\;\mbox{for all} \;\phi\in \Phi\}.$$ We will denote $\Gamma \cap U_{\reg}(\Phi)$ by $\Gamma_{\reg}(\Phi)$.
Consider the distribution $\mathcal{B}(\Phi,\Lambda)$ on $V/\Lambda$
defined via its Fourier coefficients:

\begin{equation}\label{mbseries}
\int_ Z \mathcal{B}(\Phi, \Lambda)(v) e^{-\ll 2i\pi v , \gamma \rr} =
\begin{cases} \frac{1}{\prod_{\phi \in \Phi}\ll 2i\pi \phi,\gamma \rr} & \text{if $\gamma \in \Gamma_{\reg}(\Phi)$,}\\
0 &\text{otherwise.}
\end{cases}
\end{equation}
We then have
\begin{equation}\label{mbseriess} \mathcal{B}(\Phi, \Lambda)=\sum_{\gamma \in \Gamma_{\reg}(\Phi)}  \frac{e^{\ll 2i\pi v ,\gamma \rr}}{
\prod_{\phi \in \Phi} \ll 2i\pi \phi,\gamma \rr} d_{\Lambda}v.
\end{equation}

The above sum, if not absolutely convergent, is defined as a distribution.
We call $\mathcal{B}(\Phi, \Lambda)$ the {\it multiple Bernoulli series} associated to $\Phi$ and $\Lambda$.
Clearly, it does not depend on the order of the elements $\phi$ in the
sequence $\Phi$ (it only depends on  $\Phi$ as a multiset).

\begin{remark}
The formula (\ref{mbseries}) for the Fourier coefficients of
$\mathcal{B}(\Phi,\Lambda)$ is  very similar to the
formula for the Fourier transform of the multispline distribution $T(\Phi)$ (defined in (\ref{multispline1})) on $V$:
if $\Phi$ spans a pointed cone, then the Fourier transform of $T(\Phi)$ is a generalized function on $U$
satisfying
\[\int_ V T(\Phi)(v) e^{-\ll 2i\pi v , x \rr} =\frac{1}{\prod_{\phi \in \Phi}\ll 2i\pi \phi,x\rr}
\]
on the open set  of $U$ given by $\prod_{\phi \in \Phi}\ll  \phi,x\rr \neq 0$.
\end{remark}

\medskip

We now list some properties of the distribution $\mathcal{B}(\Phi,
\Lambda)$:

$\bullet$ If $\Phi$ is the empty set, then
$$\mathcal{B}(\Phi,\Lambda)=\sum_{\gamma\in \Gamma}e^{2i\pi\langle \gamma,v\rangle } d_{\Lambda}v=\sum_{\lambda\in \Lambda} \delta_\lambda$$
is the $\delta$-distribution of the lattice $\Lambda$.

$\bullet$ If $\Phi$ contains the zero vector, then
$\mathcal{B}(\Phi,\Lambda)$ is identically equal to zero.

$\bullet$ $\mathcal{B}(\Phi,\Lambda)$ is periodic with respect to
$\Lambda.$

$\bullet$\label{aver} Let $\Lambda_1\subset \Lambda_2$. Then
$\mathcal{B}(\Phi,\Lambda_2)$ is obtained from
$\mathcal{B}(\Phi,\Lambda_1)$ by averaging over
$\Lambda_2/\Lambda_1$:
\begin{equation}\label{E:aver}
\mathcal{B}(\Phi,\Lambda_2)=
\sum_{\lambda_2\in \Lambda_2/\Lambda_1}
{\rm t}(\lambda_2)\mathcal{B}(\Phi,\Lambda_1).
\end{equation}
The above relation follows from the fact that, if $\gamma\in \Gamma_1\setminus \Gamma_2$,
then
$\sum_{\lambda_2\in \Lambda_2/\Lambda_1} e^{2i\pi\ll\gamma,\lambda_2\rr}=0.$

$\bullet$ The distribution $\mathcal{B}(\Phi,\Lambda)$ is supported
on $<\Phi>+\Lambda$.

Indeed, it is immediate to verify that  $(1-e^{\ll 2i\pi v ,\nu \rr
})\mathcal{B}(\Phi, \Lambda)=0$ for all $\nu\in \Gamma\cap
<\Phi>^{\perp}$.

$\bullet$ If $\Phi$ generates $V$, then  $\mathcal{B}(\Phi,\Lambda)$
is piecewise polynomial. We will give a  proof of  this property in
Section \ref{section:poly}.

\medskip

When the lattice $\Lambda$ is fixed, we often use the measure $d_\Lambda v$ to identify
distributions and generalized functions.

\begin{example}\label{eq:bernoulli}
Let $\Lambda=\Z \omega$, and let $\Phi_k=[\omega,\omega,\ldots,
\omega]$ where  $\omega$ is repeated $k$ times. If $k=0$, then
$\mathcal{B}(\Phi_k,\Lambda)(t\omega)= \sum_{n\in \Z}e^{2i\pi nt}dt$ is  the
$\delta$-distribution of the lattice $\Lambda$ by Poisson formula. If $k>0$, then
$$\mathcal{B}(\Phi_k,\Lambda)(t\omega)=\sum_{n\neq 0} \frac{e^{2i\pi nt}}{
(2i\pi n)^k}dt=-\frac{1}{k!} B(k,t-[t])dt,$$ where $B(k,t)$ denotes
the $k^{\text{th}}$ Bernoulli polynomial in variable $t$ and $[t]$ is the integer part of $t$. (Our
normalization for the Bernoulli polynomial is that of Maple).  In particular, for $k=1$, we have
$\mathcal{B}(\Phi_1,\Lambda)(t\omega)dt= (\frac{1}{2}-t+[t])dt$ (see
Figure \ref{berno}).

\begin{figure}
\begin{center}
  \includegraphics[width=93mm]{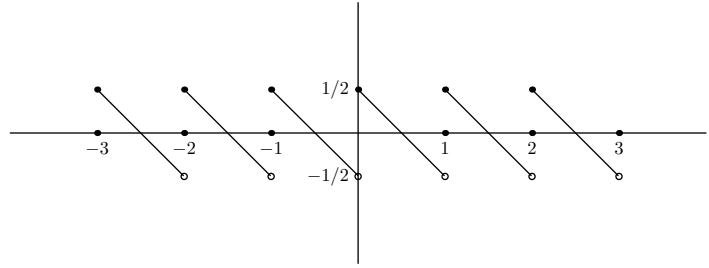}\\
  \caption{Graph of the function  $\mathcal{B}(\Phi_1,\Lambda)(t\omega)= (\frac{1}{2}-t+[t]$)}\label{berno}
\end{center}
\end{figure}

If $k>1$, the above series is absolutely convergent and
$\mathcal{B}(\Phi_k,\Lambda)(t\omega )$ is  given by integration
against a continuous function on $\R$.
\end{example}

\begin{example}\label{bernoulli2}
Let $V=\R e_1\oplus \R e_2$ with lattice $\Lambda=\Z e_1\oplus \Z
e_2$. Let $\Phi=[e_1,e_2,e_1+e_2]$. We write $v\in V$ as $v=v_1
e_1+v_2 e_2$.

We compute the  series
$\mathcal{B}(\Phi,\Lambda)=\mathcal{B}(v_1,v_2)dv_1dv_2$ where
\[\begin{array}{ll}
\mathcal{B}(v_1,v_2)&=\displaystyle \sum_{\substack{n_1\in \Z,\, n_2\in\Z\\ n_1\neq 0,\, n_2\neq 0,\, n_1+n_2\neq 0}}
\frac{e^{2i\pi(n_1v_1+n_2v_2)}}{(2i\pi n_1)(2i\pi
n_2)(2i\pi(n_1+n_2))}.
\end{array}\]

The distribution  $\mathcal{B}(\Phi,\Lambda)$ is piecewise polynomial
and periodic with respect to $\Lambda=\Z e_1+\Z e_2$.  It is thus sufficient
to write the formulae of $\mathcal{B}(v_1,v_2)$ for $0<v_1<1$ and $0<v_2<1$ (see Figure \ref{aa2}).
%If $v_1< v_2$ then $v\in \tau_1$ otherwise $v\in \tau_2$

\begin{figure}
\begin{center}
 \includegraphics[width=37mm]{tope2mod.mps}\\
 \caption{$\CT([e_1,e_2,e_1+e_2],\Z e_1 \oplus \Z e_2)$}\label{aa2}
 \end{center}
\end{figure}

\[\mathcal{B}(v_1,v_2)=\left\{\begin{array}{cl}
   -\frac{1}{6}(1+v_1-2v_2)(v_1-1+v_2)(2v_1-v_2) , & \text{if}\;v_1< v_2\\
   -\frac{1}{6}(v_1-2v_2)(v_1-1+v_2)(2v_1-1-v_2),   & \text{if}\;v_1> v_2.
       \end{array}\right.\]

We remark that $3\mathcal{B}(v_1,v_2)$ is the symplectic volume of the moduli space of flat $\text{SU}(3)$-connections
on a topological torus with one marked point $v=v_1H_{\alpha_1}+v_2H_{\alpha_2}$ where $H_{\alpha_1}$ and $H_{\alpha_2}$
denote coroots associated to simple roots $\{\alpha_1,\alpha_2\}$ of $\text{SU}(3)$.
\end{example}

\begin{example}\label{exampB2}
Let $V=\R e_1\oplus \R e_2$ with lattice $\Lambda=\Z e_1\oplus \Z
e_2$. Let $\Phi=[e_1,e_2,e_1+e_2,e_1-e_2]$. We write $v\in V$ as $v=v_1
e_1+v_2 e_2$. We compute the  series
$\mathcal{B}(\Phi,\Lambda)=\mathcal{B}(v_1,v_2)dv_1dv_2$ where
 \[\begin{array}{ll}
\mathcal{B}(v_1,v_2)&=\displaystyle \sum_{\substack{n_1\in \Z,\, n_2\in\Z\\ n_1\neq 0,\, n_2\neq 0,\, n_1+n_2\neq 0,\,n_1-n_2\neq 0}}
\frac{e^{2i\pi(n_1v_1+n_2v_2)}}{(2i\pi n_1)(2i\pi
n_2)(2i\pi(n_1+n_2))(2i\pi(n_1-n_2))}.
\end{array}\]
In the region $v_1-v_2<1, v_2<0, v_1+v_2>0$, we get
$$B(v_1,v_2)=\frac{1}{8}v_2(2v_1-1)(v_1-1-v_2)(v_1+v_2).$$
In the region $v_1>v_2, v_2>0, v_1+v_2<1$, we get
$$B(v_1,v_2)=\frac{1}{8}v_2(2v_1-1)(v_1-1+v_2)(v_1-v_2).$$
Similar computation in the region $v_1>v_2, v_2>0, v_1+v_2<1$ gives
$$B(v_1,v_2)=\frac{1}{8}v_1(2v_2-1)(v_1-1+v_2)(v_1-v_2).$$
\end{example}

\section{Recurrence relations}\label{section:recurrence}

For an element $\phi$ in $\Phi$, we associate two lists
of vectors as follows:\\

$\bullet$ We consider the list $\Phi-\{\phi\}$ in $V$ and the corresponding distribution
$\mathcal{B}(\Phi-\{\phi\}, \Lambda)$ on $V$.

$\bullet$ Consider the vector space $V_0:=V/<\phi>$, let $p$ denote the projection $V\to V_0$.
We denote the image under $p$ of the lattice $\Lambda$ in $V_0$ by $\Lambda_0$.
The dual space $U_0$ of the vector space  $V_0$ is the hyperplane $H_\phi$. The
dual lattice to $\Lambda_0$ is  the lattice $\Gamma_0=\{\gamma\in
\Gamma\,|\, \langle \gamma,\phi\rangle =0\}$.
Consider the list $\Phi_0$  of elements of $V_0$ consisting of the images  of the elements in
$\Phi-\{\phi\}$.   Observe that if $\Phi$ contains $\phi$ with multiplicity greater than $1$, then
$\Phi_0$ contains the zero vector and consequently $\mathcal{B}(\Phi_0,\Lambda_0)$ is identically zero.
%Nonzero elements in the list $\Phi_0$ identifies to a list of linear forms on $H_\phi$.

If $D$ is a  distribution  on $V_0$, we denote by $p^*D$ the
distribution on $V$ ``constant in the direction $\phi$": if
$D=b(v_0) d_{\Lambda_0}v_0$, then we define  $p^*D=b(p(v))
d_{\Lambda}v$. Thus $p^*\mathcal{B}(\Phi_0,\Lambda_0)$ is a
distribution on $V$.  We remark that, for any $\phi \in \Phi$, we have  the following equality of
sets
\begin{equation}\label{eq:sets}
\Gamma_{\reg}(\Phi-\{\phi\})=
(\Gamma_{\reg}(\Phi-\{\phi\})\cap \{\phi=0\})\cup \Gamma_{\reg}(\Phi)
\end{equation}
where the union is disjoint.

\bigskip
The main remark  of this section is the following  recurrence
relation for the distribution  $\mathcal{B}(\Phi,\Lambda)$.

\begin{proposition}
Let $\phi\in \Phi$.  Then  we have
\begin{equation}\label{induction}
\partial_\phi \mathcal{B}(\Phi,\Lambda)=\mathcal{B}(\Phi- \{\phi\},\Lambda)-p^*\mathcal{B}(\Phi_0,\Lambda_0).
\end{equation}
\end{proposition}

%To be reassured,  let us verify  the formula where $V$ is if dimension $1$, $\phi=y\omega$ and $\Lambda= \Z \omega$.
%Then $$\mathcal{B}(\Phi,\Gamma)(t\omega )=\sum_{n\neq 0}e^{2i\pi n t}{y n}.$$
%Thus $Ber(\Phi,\Gamma)(t\omega)=-\frac{1}{y} (t-[t]-1/2)$
%We see that $y\partial_t Ber(\Phi,\Gamma)(t\omega)=-1$ which is the function $-Ber(\Phi_0,\Gamma_0)$.

\begin{proof}

We fix the measures $d_\Lambda v$ and $d_{\Lambda_0} v$ and we
identify $\mathcal{B}(\Phi,\Lambda)$ and
$\mathcal{B}(\Phi_0,\Lambda_0)$ to generalized functions. Differentiating $\mathcal{B}(\Phi,\Lambda)$ in the sense of
generalized functions, we get
\[\partial_{\phi}\mathcal{B}(\Phi,\Lambda)(v)
=\sum_{\gamma \in \Gamma_{\reg}(\Phi)}  \frac{e^{\ll 2i\pi v ,\gamma
\rr}}{ \prod_{\phi' \in \Phi-\{\phi\}} \ll 2i\pi \phi',\gamma
\rr}\]
\[=\sum_{\gamma \in \Gamma_{\reg}(\Phi-\{\phi\})}
\frac{e^{\ll 2i\pi v ,\gamma \rr}}{\prod_{\phi' \in
\Phi-\{\phi\}} \ll 2i\pi \phi',\gamma \rr} - \sum_{ \gamma \in \Gamma_{\reg}(\Phi-\{\phi\}), \ll
\gamma,\phi\rr=0} \frac{e^{\ll 2i\pi v ,\gamma \rr}}{\prod_{\phi'
\in \Phi-\{\phi\}} \ll 2i\pi \phi',\gamma \rr}.\]

The last term  is constant on the line $v+\R \phi$ and identifies
with $p^*\mathcal{B}(\Phi_0,\Lambda_0)$.
\end{proof}

\section{Hyperplane arrangements and generalized series}\label{sub:hyper}

We generalize the setting of Bernoulli series.

Here we assume that the list $\Phi$ in $V_{\mathbb Q}$ does not contain the zero vector.  Then each $\phi$ in $\Phi$ determines an hyperplane
$H_\phi=\{u \in U: \ll u,\phi\rr=0\}$ in $U$.  Let $$\CH(\Phi)=\{H_\phi, \;\phi \in\Phi\}$$ be the set of hyperplanes determined by $\Phi$.
We denote the closed subset $\cup_{\phi \in \Phi} H_\phi$ of $U$ by the same notation  $\CH(\Phi)$.
When $\Phi$ is fixed, we denote $\CH(\Phi)$ simply by $\CH$, and its complement in $U$ by $U_\CH$.

We denote by $S(V)$ the symmetric algebra of $V$ and identify it with the ring of polynomial functions on $U$.
We denote by $R_{\mathcal H}$ the ring of rational functions on $U$ regular on $U_{\mathcal{H}}$, that is, the ring generated by the ring
$S(V)$ of polynomial functions on $U$ together with inverses of linear forms $\phi \in \Phi$.

The set $\Gamma_{\reg}(\Phi)$ depends only on $\mathcal H$, thus, we shall also denote it by $\Gamma_{\reg}(\mathcal H)$.
A function $g\in R_{\mathcal H}$ is well defined at $\gamma\in \Gamma_{\reg}(\mathcal H)$.

\begin{definition}
If $g\in R_{\mathcal H}$, we define the distribution $\mathcal{B}(\mathcal H, \Lambda,g)$ on $V$ by
$$\mathcal{B}(\mathcal H, \Lambda,g)=\sum_{\gamma \in \Gamma_{\reg}(\mathcal H)} g(\gamma) e^{ 2i\pi \ll v ,\gamma \rr} d_\Lambda v.$$
\end{definition}

It is easy to see that the above series converges  in the space of distributions on $V$.
The Bernoulli series $\mathcal B(\Lambda,\Phi)$ is the special case of $\mathcal B(\mathcal H,\Lambda,g)$ with
$g=\frac{1}{\prod_{\phi\in \Phi}\phi}.$

\begin{example}\label{ex:deltaone}
Let $\Lambda=\Z \omega$,  let $\CH=\{0\}$, and $g=1$.
Then    $$\mathcal{B}(\mathcal H, \Lambda,g)=\sum_{n\neq 0}e^{2i\pi nt}=-1+\delta_{\Lambda}$$

\end{example}
\medskip

Observe that if $\Lambda_1\subset \Lambda_2$, then $\mathcal{B}(\mathcal H,\Lambda_2,g)$ is obtained from
$\mathcal{B}(\mathcal H,\Lambda_1,g)$ by averaging over $\Lambda_2/\Lambda_1$:
\begin{equation}\label{E:aver2}
\mathcal{B}(\mathcal H,\Lambda_2,g)= \sum_{\lambda_2\in \Lambda_2/\Lambda_1}
{\rm t}(\lambda_2)\mathcal{B}(\mathcal H,\Lambda_1,g).
\end{equation}

\bigskip

Let $\phi\in \Phi$, then we can associate to $\phi$ the following two arrangements:

$\bullet$  $\mathcal H'=\mathcal H\setminus H_\phi$.
%in other words, in the list  of equations for $\mathcal H'$, we have  deleted {\bf all elements proportional} to $\phi$.

$\bullet$ $\mathcal H_0=\{H\cap H_\phi, H\in \mathcal H'\}$, the trace of the arrangement $\mathcal H'$ on $H_\phi$.

\bigskip

Clearly a function $g$ in $R_{\mathcal H'}$ restricts to the hyperplane $H_\phi$ in a rational function $g_0$ lying in
$R_{\mathcal H_0}$.  Thus
$\mathcal{B}(\mathcal H_0,\Lambda_0,g_0)$ is a distribution on
$H_\phi^*=V/\R \phi$.

We have the following recurrence relation for the distribution
$\mathcal{B}(\mathcal H,\Lambda,g)$ associated to an element $g\in
R_{\mathcal H'}$.

\begin{proposition}\label{prop:recug}
If $g\in R_{\mathcal H'}$, then
$$\mathcal{B}(\mathcal H,\Lambda,g)=\mathcal{B}(\mathcal H',\Lambda,g)-p^*
\mathcal{B}(\mathcal H_0,\Lambda_0,g_0).$$
\end{proposition}

\begin{proof}

From the equality (\ref{eq:sets}), we see that
the elements of $\Gamma_{\reg}(\CH')$ that are not in $\Gamma_{\reg}(\CH)$
can be identified with the elements of $\Gamma_{\reg}(\CH_0)$.
\end{proof}

\section{Piecewise polynomial behavior }\label{section:poly}

For completeness we reprove here that the distribution
$\mathcal{B}(\Phi,\Lambda)$ is piecewise polynomial when $\Phi$
generates $V$.  In fact, we prove the piecewise polynomial behavior of the series
$\mathcal{B}(\mathcal H,\Lambda,g)$ when $g$ belongs to a particular subset $G_{\mathcal H}$ of
$R_{\mathcal H}$ which we will shortly describe.

%\begin{definition}
Suppose $\Phi$ generates $V$. A subspace of $V$ generated by a subset of elements of $\Phi$ is
called  $\Phi$-admissible.
%We denote by $\CR(\Phi)$ the set of  $\Phi$-admissible subspaces of $V$.
%If $\Phi$ is fixed, we denote $\CR(\Phi)$ simply by $\CR$.
%\end{definition}
A $\Phi$-admissible  hyperplane will also be called a {\it wall}.
Let $\CH_{\text{aff}}(\Phi)$ be the set of  $\Phi$-admissible hyperplanes in $V$ together with
their translates with respect to $\Lambda$.
An element $W\in \CH_{\text{aff}}(\Phi)$  will also be called a (affine) wall.
An element $v\in V$  is said to be {\it regular} if $v$ is not on any affine wall.
We denote by  $V_{\reg,\aff}$ the open subset of $V$ consisting of regular elements.
We denote by $\CT(\Phi,\Lambda)$ the set of connected
components of  $V_{\reg,\aff}$.  An element $\tau$ of $\CT(\Phi,\Lambda)$ is called a
{\it tope}.  By definition, topes only depend on $\Lambda$ and the arrangement $\CH(\Phi)$,
and not on $\Phi$ itself; thus we will denote the set of topes indifferently by $\CT(\Phi,\Lambda)$ or
$\CT(\CH,\Lambda)$.

Suppose $ \Lambda_1\subset \Lambda_2$.  Then, topes corresponding to the system $(\Phi,\Lambda_2)$
are obtained by translating topes corresponding to $(\Phi,\Lambda_1)$ by elements of $\Lambda_2$ and taking their nonempty intersections.

\begin{example}
Let $V=\R e_1\oplus \R e_2$, $\Phi=[e_1,e_2]$ and $\Lambda=\Z e_1\oplus \Z \frac{(e_1+e_2)}{2}$.
Let $\Phi=[e_1,e_2]$.
Then, the topes in $\CT(\Phi,\Z e_1\oplus \Z e_2)$ gives a paving of $V$ by squares
(complement of bold black lines in Figure \ref{F:tope}),
and the topes in $\CT(\Phi,\Lambda)$  are obtained by  subdividing the squares into $4$ equal squares.
\end{example}

%The configuration of the topes $\CT([e_1,e_2,e_1+e_2],\Lambda)$ and $\CT([e_1,e_2,e_1+e_2,e_1-e_2],\Lambda)$ where the
%lattice $\Lambda=\Z e_1\oplus \Z e_2$ are depicted in Figure \ref{topea2} (on page $18$) parts $(a)$ and $(b)$ respectively.

\begin{figure}
\begin{center}
  \includegraphics[width=37mm]{topes.mps}\\
  \caption{$\CT(\Phi,\Z e_1\oplus \Z e_2)$ versus $\CT(\Phi,\Z e_1\oplus \Z \frac{(e_1+e_2)}{2})$}\label{F:tope}
\end{center}
\end{figure}

\begin{definition}
A function $f$ on $V_{\reg,\aff}$ is called \textit{piecewise polynomial with
respect to $\CH_{\aff}(\Phi)$ and $\Lambda $} if $f$ coincide with a polynomial
function $f^{\tau}$ on each tope $\tau$ in $\CT(\Phi,\Lambda)$.
\end{definition}

A distribution $D$ is called piecewise polynomial with
respect to $\CH(\Phi)$ and $\Lambda $ (in short
$(\CH,\Lambda)$, or equivalently $(\Phi,\Lambda)$) if it is given by integration on $V_{\reg,\aff}$ by a
piecewise polynomial function. The space of piecewise polynomial distributions with
respect to $(\CH,\Lambda)$ is invariant under translation by
$\Lambda$. More generally, if $\Lambda_1\subset \Lambda_2$, and $D$
is piecewise polynomial with respect to $(\CH,\Lambda_1)$, then ${\rm
t}(\lambda_2)D$ is piecewise polynomial with respect to
$(\mathcal H,\Lambda_2)$ for any $\lambda_2\in \Lambda_2$.

The condition for   a  distribution   $b$  to be piecewise polynomial  is {\bf stronger} than the condition  that the restriction of $b$
to any tope  is a polynomial density.  For example the $\delta$ function of the lattice $\Lambda$ restricts to $0$ on any tope $\tau$,
but is not a piecewise polynomial distribution.

\bigskip

Let $\phi\in \Phi$, and consider the two arrangements $\mathcal H'$ and $\mathcal H_0$ associated to $\phi$ as in the
previous section.  If $f$ is piecewise polynomial for $(\mathcal H',\Lambda)$, then $f$ is piecewise polynomial for $(\mathcal H,\Lambda)$.
If $f_0$ is piecewise polynomial for $(\mathcal H_0,\Lambda_0)$, then $p^*f$ is piecewise polynomial for $(\mathcal H,\Lambda)$.

\bigskip

We now prove that $\mathcal{B}(\mathcal H,\Lambda,g)$ is piecewise polynomial with respect
to $(\mathcal H,\Lambda)$ when $\Phi$ generates $V$ and $g$ is in some subspace of $R_{\mathcal H}$ that we describe now.

We may assume that all equations $\phi=0$ of the hyperplanes $H_\phi$ in $\mathcal H$ lie in $\Lambda$, we can always
achieve this by taking an appropriate multiple of $\phi$.  Thus, for what follows, we may assume that elements of $\Phi$  are in fact in $\Lambda$.

We denote by $\mathfrak{B}(\Phi)$ the set of subsets of $r$ linearly
independent elements of $\Phi$.  In other words, an element of $\mathfrak{B}(\Phi)$ is
a basis of $V$ extracted from $\Phi$.

Suppose $L$ is a sequence of elements of $\Phi$ (possibly with multiplicities) generating $V$.
Define $$\theta(L)(x)=\frac{1}{\prod_{\alpha\in L}\langle \alpha,x\rangle},$$
a function in $R_{\mathcal H}$.

Since $\theta(L)$ will change by a scalar multiple when elements of $L$ are scaled,
we may define the following space which depends only on $\CH$.

\begin{definition}\label{defgphi}
Let $G_{\mathcal H}$ be the subspace of $R_{\mathcal H}$ generated by all rational functions of the form $\theta(L)$.
\end{definition}

We recall the following description of $G_{\mathcal H}$.

\begin{lemma}\label{ind}
Any $\theta(L)$ may be written as a linear combination of elements
$\theta(\sigma,{\bf n})=\frac{1}{\alpha_{i_1}^{n_1}\cdots \alpha_{i_r}^{n_r}}$ where
$\sigma:=[\alpha_{i_1},\ldots, \alpha_{i_r}] \in \mathfrak{B}(\Phi)$ is a basis extracted from $\Phi$
and ${\bf {n}}=[n_1,n_2,\ldots,n_r]$ is a sequence of positive integers.
\end{lemma}

\begin{proof}
By induction on the number of elements of $L$, we need to prove that the assertion holds for rational fractions of the form
$\theta(\sigma,{\bf n})\frac{1}{\alpha^N}$ with $\sigma=[\alpha_{1},\ldots, \alpha_{r}]$ is a basis of $V$.
We write $\alpha=\sum_{i=1}^rc_i \alpha_i$. Using
the relation,
$$\theta(\sigma,{\bf n})\frac{1}{\alpha^N}=\frac{\alpha}{\alpha_1^{n_1}\cdots \alpha_r^{n_r}}\frac{1}{\alpha^{N+1}}
=\sum_{i, c_i\neq 0} c_i \frac{1}{\alpha_1^{n_1}\cdots
\alpha_i^{n_i-1}\cdots \alpha_r^{n_r}}\frac{1}{\alpha^{N+1}},$$
we decrease the number $|{\bf {n}}|=n_1+\cdots+n_r$. When one of the $n_i$ in the sum becomes $0$,
the corresponding term is of the required form associated to the basis $\sigma_i=\sigma \cup \{\alpha\}\setminus \{\alpha_i\}$.
\end{proof}

\begin{proposition}\label{compute}
If $g\in G_{\mathcal H}$, the
 distribution $$\mathcal{B}(\mathcal H,\Lambda,g)(v)=\sum_{\gamma\in \Gamma_{\reg}(\mathcal H)}g(\gamma) e^{2i\pi\langle \gamma,v\rangle }$$
is a piecewise polynomial distribution (with respect to the system
$(\mathcal H, \Lambda)$).
\end{proposition}

\begin{proof}
We prove the proposition by induction on the number of elements in $\mathcal H$.  Using Lemma \ref{ind}, it
suffices to prove the proposition for $g$ of the form $g=\theta(\sigma, {\bf n})$ for $\sigma \in \mathfrak{B}(\Phi)$ and ${\bf n}$ a sequence
of positive integers.

We first assume that $\Phi$ consists of independent elements $\alpha_1,\alpha_2,\ldots,\alpha_r$ possibly with
multiplicities.  Consider an element $g=\theta(L)$ of $G_{\mathcal H}$ for
$$L=[\alpha_1,\ldots, \alpha_1,\alpha_2,\ldots, \alpha_2, \ldots, \alpha_r,\ldots, \alpha_r]$$  where
$\alpha_i$ appears with multiplicity $k_i$ in $\Phi$.

Let $\Lambda'=\Z\alpha_1\oplus \Z \alpha_2 \oplus\cdots \oplus  \Z
\alpha_r$; clearly $\Lambda'$ is a sublattice of $\Lambda$. We
choose coordinates $t=\sum_i t_i \alpha_i$. Then
$\mathcal{B}(\mathcal H,\Lambda',g)= B(t) $ with
\begin{equation}\label{degree}
B(t)=(-1)^r \prod_{i=1}^r\frac{1}{k_i !} B(k_i,t_i -[t_i]).
\end{equation}
The function $B(t)$ is  a polynomial function on each parallelogram
translated from the  parallelogram $\sum_{i=1}^r [0,1]\alpha_i$. By
equation (\ref{E:aver2}),  the distribution
$\mathcal{B}(\mathcal H,\Lambda,g)$ is obtained by averaging
$\mathcal{B}(\mathcal H,\Lambda',g)$ over $\Lambda/\Lambda'$. Thus
$\mathcal{B}(\mathcal H,\Lambda,g)$ is piecewise polynomial with respect to
$(\mathcal H, \Lambda)$.

We now consider the general case, where the cardinality of $\mathcal H$ is greater than $r$.   In this case there exists
$\phi \in \Phi$ with the property that, for the hyperplane
arrangements $\CH'$ and $\CH_0$ associated to  $\phi$, we have $g\in G_{\mathcal H'}$ and
$g_0$ (the restriction of $g$ to $\phi=0$) is in $G_{\mathcal H_0}$.  Then, using Proposition \ref{prop:recug}, which states
$$\mathcal{B}(\mathcal H,\Lambda,g)=
\mathcal{B}(\mathcal H',\Lambda,g)-p^*\mathcal{B}(\mathcal H_0,\Lambda_0,g_0),$$ we conclude
by induction that $\mathcal{B}(\mathcal H,\Lambda,g)$ is piecewise polynomial with respect to $(\CH,\Lambda)$.
\end{proof}

The Bernoulli series $\mathcal{B}(\Phi,\Lambda)$  is equal to $\mathcal{B}(\mathcal H,\Lambda,g)$ with
$g=\frac{1}{\prod_{\phi\in \Phi} \phi}$;
it is an element of $G_{\mathcal H}$ for we assumed that $\Phi$ spans $V$.  Thus, by Proposition \ref{compute},
we immediately obtain that $\mathcal{B}(\Phi,\Lambda)$ is a piecewise polynomial density.

\begin{corollary}

For any $f\in R_\CH$, the distribution $\mathcal{B}(\mathcal H,\Lambda,f)(v)$ restricts to a tope $\tau$ as a polynomial density.
\end{corollary}
\begin{proof}

 Let $f\in  R_\CH$. We can write $f$ as $Pg$ where $P$ is a polynomial and $g\in G_\CH$.
 Then the distribution  $\mathcal{B}(\mathcal H,\Lambda,f)(v)$
 is obtained by applying the differential operator $P(\partial_v)$ to the distribution  $\mathcal{B}(\mathcal H,\Lambda,g)(v)$.

 This  differentiation is in the distribution sense so that it may produce distributions supported on admissible hyperplanes, but
 on an open tope $\tau$, we obtain a polynomial density.
 \end{proof}

\begin{definition}
Given a tope $\tau$ in $\CT(\Phi,\Lambda)$, we denote by
$\Ber(\Phi,\Lambda,\tau)$ the polynomial function  on $V$ such that
the restriction of $\mathcal{B}(\Phi,\Lambda)$ to  $\tau$ coincides
with the restriction of $\Ber(\Phi,\Lambda,\tau)(v) d_{\Lambda}v$
on $\tau$.
\end{definition}

By the above proof, we see that the polynomial
$\Ber(\Phi,\Lambda,\tau)$  is of degree equal to the number of elements in $\Phi$.

\bigskip

The fact that $\mathcal{B}(\Phi,\Lambda)(v)$ is a periodic
distribution  on $V$  implies immediately the following periodicity
formula. For any $\lambda\in \Lambda$ and $v\in V$,
\begin{equation}\label{periodic}
\Ber(\Phi,\Lambda,\tau+\lambda)(v+\lambda)=
\Ber(\Phi,\Lambda,\tau)(v).
\end{equation}

If $\nu$ is a connected subset of $V$ contained in the open set of
regular elements, we denote by $\Ber(\Phi,\Lambda,\nu)$ the
polynomial $\Ber(\Phi,\Lambda,\tau(\nu))$ where $\tau(\nu)$ is the
unique tope containing $\nu$.

Let $\tau \in \CT(\Phi, \Lambda)$ and $\phi\in \Phi$.
If  $v_0\in V_0=V/<\phi>$ is the projection of $v\in \tau$, then $v_0$ is not on any affine wall in $V_0$. Indeed the reciproc image of an affine wall in $V_0$ is an affine wall in $V$.
We denote by $\tau_0$ the unique tope in $V_0$ containing the projection of $\tau$.

\medskip
Equation (\ref{induction}) implies the following relations.
\medskip

If $\Phi- \{\phi\}$ generates $V$, then
\begin{equation}\label{phiphi1}
\partial_\phi \Ber(\Phi,\Lambda,\tau)=\Ber(\Phi- \{\phi\},\Lambda,\tau)-\Ber(\Phi_0,\Lambda_0,\tau_0).
\end{equation}

If $\Phi$ generates $V$, but $\Phi- \{\phi\}$ does not generate $V$, then
\begin{equation}\label{phiphi2}
\partial_\phi \Ber(\Phi,\Lambda,\tau)=-\Ber(\Phi_0,\Lambda_0,\tau_0).
\end{equation}

\begin{remark}
By using reduction to independent variables and the explicit formula  (\ref{degree}), we obtain also a way to compute
$\Ber(\Phi,\Lambda,\tau)$.  This can be applied not too painfully when the number of elements in $\Phi$ is small.
However, the residue formula   due to A.~Szenes \cite{sze1} to compute $\Ber(\Phi,\Lambda,\tau)$  is very efficient
when $\Phi$ is large, provided the dimension of $V$ is relatively small. We will give examples of computations of volumes of moduli spaces using Szenes formula in a next article.
\end{remark}

\section{An Euler-MacLaurin formula}

This section is independent of the rest of the article.

Assume that $\Phi$ generates $V$.  Using the Lebesgue measure associated to the lattice $\Lambda$, we identify
$\mathcal{B}(\Phi,\Lambda)(v)$ to a piecewise polynomial function on $V$.

Let us denote by $\CR$ the set of $\Phi$-admissible subspaces of $V$.
Then $\s=V$ and $\s=\{0\}$ are the maximum and minimum elements of the partially ordered set $\CR$.
If  $\s$ is a $\Phi$-admissible subspace of $V$, we
denote by $\Phi\setminus \s$ the sequence of elements in $\Phi$ not lying in the space $\s$.

The projection of the list $\Phi\setminus \s$ on $V/\s$
will be denoted by $\Phi/\s$. The image of the lattice $\Lambda$ in $V/\s$ is a lattice in $V/\s$.
If $\Phi$ generates $V$, $\Phi/\s$ generates $V/\s$.
Using the projection $V\to V/\s$, we identify the piecewise  polynomial function ${\mathcal B}(\Phi/\s,\Lambda/\s)$ on $V/\s$  to  a
piecewise  polynomial function on $V$ constant along the affine spaces $v+\s$.
Then,  $$\Gamma_{\rm reg}(\Phi/\s):=\Gamma\cap U_{\rm reg}(\Phi/\s)$$ is  the set of elements $\gamma\in \Gamma$  satisfying
$\ll \gamma,s\rr=0$ for all $s\in \s$ and $\ll \gamma,\phi\rr\neq 0$ for  all $\phi\in \Phi\setminus \s$.

We lift functions on $V/\s$ to functions on $V$ by the canonical projection.
Thus ${\mathcal B}(\Phi/\s,\Lambda/\s)$ is the  function on $V$ given by the series (convergent in the sense of generalized functions)
$$\sum_{\gamma\in \Gamma_{\rm reg}(\Phi/\s)} \frac{e^{2i\pi\ll v,\gamma \rr}}{\prod_{\phi\in \Phi\setminus \s} 2i\pi \ll \phi,\gamma\rr}.$$
This function is periodic with respect to the lattice $\Lambda$, piecewise polynomial on $V$ (relative to ($\Phi$,$\Lambda$)) and constant along $v+\s$.
We denote it simply by $\mathcal B(\Phi/\s)$ leaving its dependence on the lattice $\Lambda$ implicit.
If $\s=V$, the function  ${\mathcal B}(\Phi/\s)$  is identically equal to $1$; if $\s=\{0\}$, then  we obtain the multiple Bernoulli series
$\mathcal B(\Phi,\Lambda)$.

\begin{theorem}\label{eumacl}
Let $f$ be a smooth function on $V$, rapidly decreasing with rapidly decreasing derivatives.
Then,
$$\sum_{\lambda\in \Lambda} f(\lambda)=\sum_{\s\in \CR} (-1)^{|\Phi\setminus \s|}
\int_V {\mathcal B}(\Phi/\s)(v) (\prod_{\phi\in \Phi\setminus \s} \partial_\phi) f  (v) dv.$$
\end{theorem}

\begin{remark}
The term corresponding to $\s=V$ in the above sum gives the term
$\int_V f(v) dv$.
Thus we may also write the formula as
$$\sum_{\lambda\in \Lambda}f(\lambda)-\int_V f(v)dv=\sum_{\s\neq V} (-1)^{|\Phi\setminus\s|} \int_V {\mathcal B}(\Phi/\s)(v)
(\prod_{\phi\in \Phi\setminus \s} \partial_\phi) f  (v) dv.$$

All the sets $\Phi\setminus \s$ entering in this formula are `long', that is, their complement in $\Phi$ do not generate $V$. In particular
they contain a `cocircuit'.  This formula has been used in \cite{Vergnebox} to obtain a formula for the semi-discrete convolution with the Box Spline.
\end{remark}

\begin{proof}
Let
$$\hat f(y)=\int_{V} e^{2i\pi \langle y,x\rangle} f(x)dx.$$
By Poisson formula
$$\sum_{\lambda \in \Lambda} f(\lambda) =\sum_{\gamma\in \Gamma} \hat f(\gamma).$$

We  group together the terms in $\Gamma$ belonging to $\s^{\perp}$ for $\s \in \CR$. More precisely,
the lattice $\Gamma$ is a disjoint union  over the $\s\in \CR$ of the sets
$$\Gamma_{\reg}(\Phi/\s)=\{\gamma \in \s^{\perp}\cap \Gamma| \ll \phi,\gamma\rr\neq 0\; \text{for all}\;\phi\in \Phi\setminus \s\}.$$
Now in the generalized sense
$$\sum_{\gamma\in \Gamma_{\reg}(\Phi/\s)} \hat f(\gamma)=\int_V \sum_{\gamma\in  \Gamma_{\reg}(\Phi/\s)}e^{2i\pi\ll v,\gamma\rr} f(v)dv$$
$$=\int_V \partial_{\Phi\setminus \s} {\mathcal B}(\Phi/\s)(v) f(v)$$
and we obtain the statement in the theorem.
\end{proof}

\section{Wall crossing}\label{section:Jump}

In this section we again assume that $\Phi$ generates $V$. Under this assumption, we compare the polynomials
$\Ber(\Phi,\Lambda,\tau)$ associated to two adjacent topes of $\CT(\Phi, \Lambda)$ separated by an hyperplane $W$.
We remark that due to the periodicity property of $\mathcal{B}(\Phi,\Lambda)$ it suffices to consider jumps over an hyperplane
$W$ passing through the origin.

If $D_1$ and $D_2$  are two distributions on $V$ with supports $S_1$ and $S_2$ with the property that  for any $v\in V$
the intersection of $v-S_1$ and $S_2$
lies in a  compact set, then the convolution $D_1*D_2$ is well
defined.

We recall the definition of multispline. Let $X=[v_1,v_2,\ldots,v_m]$ be a sequence of non-zero vectors in $V$. We will first
consider the case where $X$ spans a pointed cone. The {\it multivariate spline $T(X)$} is the tempered distribution defined on
test functions $f$ by:
\begin{equation}\label{multiva}
\langle T(X)\,|\,f\rangle = \int_0^\infty\cdots\int_0^\infty
f(\sum_{i=1}^mt_i v_i)dt_1\cdots dt_m.
\end{equation}

If $X$ spans $V$, we may interpret  $T(X)$ as a function on $V$ supported in the cone $C(X)$ generated by $X$. This function is
piecewise polynomial. If $v\in X$, then $\partial_v T(X)=T(X-\{v\})$.  When $X$ is the empty set, $T(X)=\delta_0$.

We now consider the case where the elements of  $X$  do not necessarily lie in a half-space. We introduce a polarization of $X$ given by a
vector $u$ in $U$. Let $u\in U$ be a vector that is nonzero on all
elements of $X$. We will then say that the vector $u$ is {\em polarizing} for $X$.
Divide  the list   $X$ into two lists $A$ and $B$, the lists  of positive and negative vectors on $u$ respectively. We then define
$$T(X,u)=(-1)^{|B|}T([A, -B]) .$$

\begin{example}\label{eq:spline}
With the notation of Example \ref{eq:bernoulli},
$$
T(\Phi_k,\omega^*)=\begin{cases} &\frac{t^{k-1}}{(k-1)!}\;\;\text{if $t>0$,}\\
&0 \;\;\;\;\text{if $t<0$}\end{cases}\;\;\text{and}\;\;T(\Phi_k,-\omega^*)=
\begin{cases}&0 \;\;\;\;\text{if $t>0$,}\\
&-\frac{t^{k-1}}{(k-1)!} \;\; \text{if $t<0$.}\\
\end{cases}$$
\end{example}

\bigskip

We return to our set up.  Let $\Phi$ be a sequence of  nonzero vectors in $V$, spanning $V$.
Let $W$ be a $\Phi$-admissible hyperplane. Let $E\in \Gamma$ be an
equation of this hyperplane, where $E$ is a primitive vector in
$\Gamma$; this fixes $E$ up to sign. The lattice $\Lambda$ is fixed,
and we write simply $dv$ instead of $d_\Lambda v$. Similarly we denote by $dh$ the density determined by $\Lambda\cap W$.
As $E$ does not vanish on any element of $\Phi\setminus W$, we may define $T(\Phi\setminus W,E)$ as above; it is a distribution supported on $E\geq 0$.
Let $p$ be a polynomial density on $W$.
Then, the convolution $p*T(\Phi\setminus W,E)$ is well defined and it is supported on $E\geq 0$.
Similarly, $p*T(\Phi\setminus W,-E)$  is supported on $E\leq 0$.
It is easily proven (see \cite{bover}) that  $p*T(\Phi\setminus W,E)-p*T(\Phi\setminus W,-E)$ is given by integration
against a {\it polynomial} density. We thus define the polynomial
${\rm Pol}(p,\Phi\setminus W,E)$ by the equation
$${\rm Pol}(p,\Phi\setminus W,E)(v)dv=p*T(\Phi\setminus W,E)-p*T(\Phi\setminus W,-E).$$

The following properties of  ${\rm Pol}(p,\Phi\setminus W,E)$ follow directly from the above equation.
\begin{lemma}\label{L:derPol}
Let $\Psi=\Phi\setminus W$.

(a)Let $\psi \in \Psi$.  Then,
\[\partial_{\psi}{\rm Pol}(p,\Psi,E)={\rm Pol}(p,\Psi-\{\psi\},E).\]

(b) If $\Psi=[\psi]$, then for $h\in W$ and $t\in \R$,
\[{\rm Pol}(p, \{\psi\},E)(h+t\psi)=\frac{f(h)}{\ll \psi, E\rr}\]
if $p(h)=f(h)dh.$

(c) If $|\Psi|>1$, then the restriction of ${\rm Pol}(p, \Psi,E)$ to
$W$ vanishes of order $|\Psi|-1$.
\end{lemma}

The following  one dimensional residue formula for
${\rm Pol}(p,\Phi\setminus W,E)$ is given in \cite{bover}. It is useful in computing the convolutions.
We write $p(h)=f(h) dh$ where $f$ is a polynomial function on the hyperplane $W$.

\begin{lemma}\label{defPol}
Let $P$ be a polynomial function on  $V$ extending $f$. Then, for $v\in V$,
\[{\rm Pol}(p,\Phi\setminus W ,E)(v)=\Res_{z=0} \left(\left(P(\partial_x)\cdot
\frac{e^{\langle v,x+z E\rangle }}{\prod_{\phi\in \Phi\setminus W}\langle \phi,x+zE\rangle }\right)_{x=0}\right).\]
\end{lemma}

\medskip

\begin{theorem}\label{T:jumpber}
Let $\tau_1$ and $\tau_2$ be two adjacent topes in $\CT(\Phi, \Lambda)$ separated by the
hyperplane $W$ defined by $E$ with $\ll v,E\rr>0$ for any $v \in \tau_1$. Denote by $\tau_{12}$ the tope in $\CT(\Phi \cap
W,\Lambda\cap W)$ containing $\overline{\tau_1} \cap
\overline{\tau_2}$ in its closure. Let
$\Ber^{\tau_{12}}:=\Ber(\Phi\cap W,\Lambda \cap W,\tau_{12})dh$ be the
polynomial  density  on $W$ determined by $\tau_{12}$.
 Then,
\begin{equation}\label{E:jump} (\Ber(\Phi,\Lambda,\tau_1)-\Ber(\Phi,\Lambda,\tau_2))dv=
\Ber^{\tau_{12}} *T(\Phi\setminus W,E)-
\Ber^{\tau_{12}} *T(\Phi\setminus W,-E) .\end{equation}
\end{theorem}

\begin{remark}
Formula (\ref{E:jump}) is very similar to jump formulae for volume of reduced spaces in Hamiltonian geometry.
Indeed if $\mu:M\to \mathfrak t^*$ is a proper moment map associated to an Hamiltonian action of a  torus $T$,
then the set of regular values of $\mu$ is the complement of a certain number of affine hyperplanes. On each
connected component, volumes of reduced spaces $M_{red}(v):=\mu^{-1}(v)/T$ are given by polynomial functions of $v$.
When crossing a wall, the variation of these polynomials follow the same jump scheme as in equation (\ref{E:jump}):
they are determined by a polynomial volume function associated to a smaller Hamiltonian manifold $M_0$ and weights of
the normal bundle of $M_0$ in $M$ (\cite{par}). In particular, when the sequence $\Phi$ is comprised of
positive coroots of a compact connected Lie group $G$ with multiplicity
$2g-1$ and $\Lambda$ is the coroot lattice of $G$, the polynomials $\Ber(\Phi,\Lambda,\tau)$ describe (up to some
normalization)  the symplectic volume of the moduli space of flat $G$-connections on
Riemann surface of genus $g$ with one boundary component, around
which the holonomy is determined by $v$. These moduli spaces are reduced spaces $M_{red}(v)$ of an Hamiltonian action.
\end{remark}

\begin{proof}
We will first verify the claim for the case where there is only one
vector $\phi$ in $\Phi$ that is not contained in $W$.

Let $\Lambda_0=\Lambda\cap W$.
We consider the lattice $\Lambda_b=\Lambda_0 \oplus \Z \phi.$
By formula (\ref{E:aver}),
\begin{equation}\label{E:aver2}\mathcal{B}(\Phi,\Lambda)(v)=\sum_{\lambda_j\in
\Lambda/\Lambda_b}\mathcal{B}(\Phi,\Lambda_b)(v+\lambda_j).\end{equation}
For $t$ in a small neighborhood of zero, let $\tau_1$ (respectively $\tau_2$) denote the tope containing the
open set of $v=h+t \phi$ for $h\in W$ lying in a relatively compact open subset of $\tau_{12}$ and $t>0$
(respectively $t<0$).

We may express a representative of a non-zero $\lambda_j \in \Lambda/\Lambda_b$ as
$\lambda_j=h_j+t_j\phi$ for $h_j \in \Lambda_0$ and $t_j \notin \Z$.
As the lattice $\Lambda_b$ is product of lattices, we have
$$\mathcal{B}(\Phi,\Lambda_b)(h+t  \phi)=\mathcal{B}(\Phi \cap
W,\Lambda_0)(h) (-t+[t]+\frac{1}{2})dt.$$ Observe that the jump in
the function
$\mathcal{B}(\Phi,\Lambda_b)(v+\lambda_j)=\mathcal{B}(\Phi,\Lambda_b)(h+h_j+(t+t_j)
\phi)$ as $t$ changes sign in a small neighborhood of zero is
precisely zero for the nontrivial representative $\lambda_j$ since
$t_j$ is not integral. Thus the only contribution to the jump comes
from the trivial $\lambda_j$ in the sum of equation (\ref{E:aver2}).
We get
\begin{multline*}
\Ber(\Phi,\Lambda,\tau_1)(v)-\Ber(\Phi,\Lambda,\tau_2)(v)=\\
\frac{1}{\langle E,\phi\rangle }\Ber(\Phi \cap W,\Lambda_0)(h)((-t+\frac{1}{2}-(-t-\frac{1}{2})).\end{multline*}

The convolution product in this case is just the product in coordinates so that
$$\Ber(\Phi\cap W,\Lambda_0,\tau_{12})*T(\{ \phi\},E)-\Ber(\Phi\cap W,\Lambda_0,\tau_{12})(h)
*T(\{ \phi\},-E)$$ is equal  at the point $(h,t)$ to $\Ber(\Phi\cap W,\Lambda_0,\tau_{12})(h)$
and hence we obtain the claimed formula.

Now consider the case where there are several elements of $\Phi$ that
do not lie in $W$. Let $\phi$ be a vector in $\Phi\setminus W$.  Let
$\Phi'=\Phi-\{\phi\}$; $\Phi'$ still generates $V$ and
$\Phi' \cap W=\Phi \cap W$. Equation
(\ref{induction}) implies that in this case $\mathcal{B}(\Phi,\Lambda)$ is
continuous on $W$: indeed  the derivative in the direction $\phi$ is
a piecewise  polynomial function.

Let $\tau_1'$ and $\tau_2'$ be the topes
of $\Phi'$ containing $\tau_1$ and $\tau_2$ respectively.  They are adjacent with
respect to $W$ and $\Ber^{\tau_{12}}=\Ber^{\tau_{12}'}$. Using
equation (\ref{phiphi1}), we have
\[\partial_{\phi}\Ber(\Phi,\Lambda,\tau_1)
-\partial_{\phi}\Ber(\Phi,\Lambda,\tau_2)=\Ber(\Phi',\Lambda,\tau_1')-\Ber(\Phi',\Lambda,\tau_2').\]
Indeed the topes $\tau_1$ and $\tau_2$ give the same tope $\tau_0$ under projection onto $V_0=V/<\phi>$.

By Lemma~\ref{L:derPol} part $(a)$,\[\partial_{\phi}{\rm
Pol}(\Ber^{\tau_{12}},\Phi\setminus W,E)={\rm
Pol}(\Ber^{\tau'_{12}},\Phi'\setminus W,E).\] Denote by ${\rm
Leq}(\Phi)$ the left hand side, and by ${\rm Req}(\Phi)$ the right
hand side of equation (\ref{E:jump}). By induction, we have
$\partial_{\phi}({\rm Leq}(\Phi)-{\rm Req}(\Phi))=0$.  Thus, the
polynomial function is constant in the direction of $\R\phi$.  The
left hand side vanishes on $W$ by the continuity of
$\mathcal{B}(\Phi, \Lambda)$ on $W$.    Hence the claim.
\end{proof}

We now demonstrate the theorem  with various examples.

\begin{example} Recall the data of Example \ref{eq:bernoulli}.  Let $\tau_1$ and $\tau_2$ be two adjacent
topes defined by inequalities $0<t<1$ and $-1<t<0$ respectively.
By Theorem \ref{T:jumpber} and Example \ref{eq:spline},
$$\Ber(\Phi,\Lambda,\tau_1)(t\omega)-\Ber(\Phi,\Lambda,\tau_2)(t\omega)=\frac{t^{k-1}}{(k-1)!},$$ which is indeed
equal to $-\frac{1}{k!}B(k,t)+\frac{1}{k!}B(k,t+1)$ as it can be
seen from the explicit expression of
$\mathcal{B}(\Phi_k,\Lambda)(t\omega)$ in Example
\ref{eq:bernoulli}.
\end{example}

\begin{example} Recall the data of Example \ref{bernoulli2}.
Let $\tau_1$ and $\tau_2$ be the two adjacent topes separated by the hyperplane
$W=\R(e_1+e_2)$ (see figure \ref{topea2}(a)).  Then $E=-e^1+e^2$.
\begin{figure}
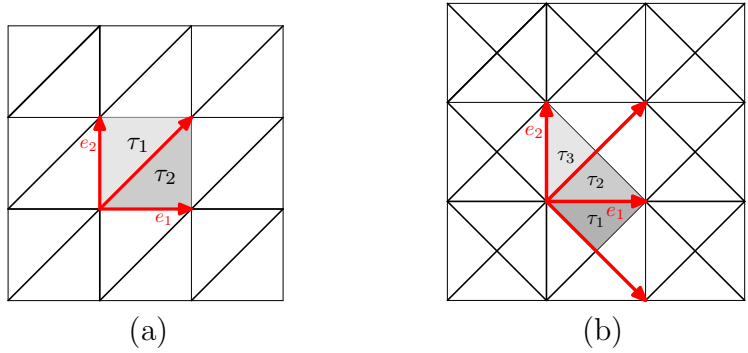

\begin{center}
 \includegraphics[width=37mm]{tope2mod.mps}\hspace{2cm}
\includegraphics[width=40mm]{tope3mod.mps}\\
  (a)\hspace{5.5cm}(b)
 \caption{$\CT([e_1,e_2,e_1+e_2],\Lambda)$\;versus $\CT([e_1,e_2,e_1+e_2,e_1-e_2],\Lambda)$ with
 $\Lambda=\Z e_1 \oplus \Z e_2$}\label{topea2}
 \end{center}
\end{figure}

We express $v=v_1e_1+v_2e_2 \in V$ as $v=v_1(e_1+e_2)+(v_2-v_1)e_2$ and $x\in U$ as $x=x_1e^1+x_2(-e^1+e^2)$.
Using Example~\ref{eq:bernoulli}, with $\Phi \cap W=e_1+e_2$ and $\Lambda \cap W=\Z(e_1+e_2)$
at $v_1(e_1+e_2)$, we have
$$
\begin{array}{ll}
\Ber^{\tau_{12}}&=\Ber(\Phi \cap W,\Lambda \cap W,\tau_{12})(v_1(e_1+e_2))\\
&=\frac{1}{2}-v_1
\end{array}$$
In the above coordinates of $U$, the operator
$\Ber^{\tau_{12}}(\partial_{x})=\frac{1}{2}-\partial_{x_1}$ under the
identification $P(\partial_x)e^{\ll a,x\rr}=P(a)e^{\ll a,x\rr}$. Then,
\[\begin{array}{ll}
{\rm Pol}(\Ber^{\tau_{12}},\Phi,E)(v)&=\Res_{z=0}\left(
\left(\Ber^{\tau_{12}}(\partial_x)\cdot \frac{e^{\ll v,x+z E\rr
}}{\prod_{\phi\in \Phi\setminus W}\ll
\phi,x+zE\rr }\right)_{x=0}\right)\\
&=\Res_{z=0}
\left(\left
((\frac{1}{2}-\partial{x_1})\cdot\frac{e^{v_1x_1+(v_2-v_1)x_2+(v_2-v_1)z}}{(x_1-x_2-z)(x_2+z)}\right)_{x=0}\right)
\\
&=\Res_{z=0}\left(\left((\frac{1}{2}-\partial{x_1})\cdot\frac{e^{v_1x_1+(v_2-v_1)z}}{(x_1-z)z}\right)_{x_1=0}\right)\\
&=\frac{1}{2}(1-v_1-v_2)(v_1-v_2),
\end{array}\]
which is indeed the jump $\Ber(\Phi,\Lambda,\tau_1)(v)-\Ber(\Phi,\Lambda,\tau_2)(v)$ as it can be seen from the
explicit expression in Example \ref{bernoulli2}.
\end{example}

\begin{example} Recall the data of Example \ref{exampB2}.

{\it (a) Jump over the wall $W=\R e_1$:}
Then $E=e^2$, $\Phi \cap W=\{e_1\}$ and $\Lambda \cap W=\Z e_1$ (see figure \ref{topea2}(b)).
\[\begin{array}{ll}
{\rm Pol}(Ber^{\tau_{21}},\Phi,E)(a)&=\Res_{z=0}
\left(Ber^{\tau_{21}}(\partial_x)\cdot \frac{e^{\ll a,x+z e^2\rr
}}{\prod_{\phi\in \Phi\setminus \Phi \cap W}\ll \phi,x+zE\rr }\right)_{x=0}\\
&=\res_{z=0}\left((-\partial_{x_1}+\frac{1}{2})
\cdot\frac{e^{v_1x_1+v_2x_2+v_2z}}{(x_2+z)(x_1+x_2+z)(x_1-x_2-z)}\right)_{x_2=0}\\
&=\frac{1}{4}v_2^2(2v_1-1),
\end{array}\]
which is indeed the jump $\Ber(\Phi,\Lambda,\tau_2)(v)-\Ber(\Phi,\Lambda,\tau_1)(v)$ as it can be seen from the
explicit expression in Example \ref{exampB2}.

\medskip

{\it (b) Jump over the wall $W=\R(e_1+e_2)$:}
Then $E=-e^1+e^2$.  We express $v=v_1e_1+v_2e_2 \in V$ as $v=v_1(e_1+e_2)+(v_2-v_1)e_2$ and $x\in U$ as
$x=x_1e^1+x_2(-e^1+e^2)$.  Using Example~\ref{eq:bernoulli}, with $\Phi \cap W=e_1+e_2$ and $\Lambda \cap W=\Z(e_1+e_2)$
at $v_1(e_1+e_2)$, we have
$\Ber^{\tau_{23}}=\Ber(\Phi \cap W,\Lambda \cap W,\tau_{23})(v_1(e_1+e_2))
=\frac{1}{2}-v_1$. Then,
\[\begin{array}{ll}
{\rm Pol}(\Ber^{\tau_{23}},\Phi,E)(v)
&=\Res_{z=0}
\left(\left
((\frac{1}{2}-\partial{x_1})\cdot\frac{e^{v_1x_1+(v_2-v_1)x_2+(v_2-v_1)z}}{(x_1-x_2-z)(x_2+z)(x_1-2x_2-2z)}\right)_{x=0}\right)
\\
&=-\frac{1}{8}(v_1-1+v_2)(v_1-v_2)^2,
\end{array}\]
which is indeed the jump $\Ber(\Phi,\Lambda,\tau_3)(v)-\Ber(\Phi,\Lambda,\tau_2)(v)$ as it can be seen from the
explicit expression in Example \ref{exampB2}.
\end{example}

\section{The affine case}\label{section:affine}

This section generalizes previous results to the affine case. Results proven here are not needed for the following section.

Let $\Phi=[\phi_1,\ldots, \phi_N]$ be a list of elements of
$V_{\mathbb Q}$ and let ${\bf z}=[z_1,z_2,\ldots, z_N]$ be a list of complex
numbers.  We consider the augmented list $\tilde
\Phi:=[[\phi_1,z_1],\ldots, [\phi_N,z_N]]$ and define
$$\Gamma_{\reg}(\tilde \Phi)=\Gamma \cap U_{\reg}(\tilde\Phi)$$ where
$$U_{\reg}(\tilde \Phi)=\{u \in U |\; \langle
\phi_j,u\rangle +z_j\neq 0\;\mbox{for all}\;j\}.$$

\begin{definition}  The \textit{affine} multiple Bernoulli series is the distribution
$$\mathcal{B}(\tilde \Phi,\Lambda)=\sum_{\gamma\in \Gamma_{\reg}(\tilde \Phi)}
\frac{e^{2i\pi \langle v,\gamma\rangle }}{ \prod_{j=1}^N 2i\pi (\langle \phi_j,\gamma\rangle +z_j)} d_{\Lambda}v.$$
\end{definition}

\medskip

The distribution $\mathcal{B}(\tilde \Phi,\Lambda)$ has the following properties, similar to its nonaffine counterpart:

$\bullet$ If $\Phi$ is the empty set, then $\mathcal{B}(\tilde \Phi,\Lambda)$ is the $\delta$-distribution of the lattice $\Lambda$.

$\bullet$  If $\Lambda_1\subset \Lambda_2$. Then $\mathcal{B}(\tilde \Phi,\Lambda_2)$ is obtained from
$\mathcal{B}(\tilde \Phi,\Lambda_1)$ by averaging over $\Lambda_2/\Lambda_1$:
\begin{equation}\label{E:aver3}
\mathcal{B}(\tilde\Phi,\Lambda_2)=
\sum_{\lambda_2\in \Lambda_2/\Lambda_1}
{\rm t}(\lambda_2)\mathcal{B}(\tilde \Phi,\Lambda_1).
\end{equation}

\medskip

In the special case $z_j=\langle \phi_j,z\rangle $ for $z\in U_\C$, it is more natural to
consider the distribution
$$\Eis(\Phi,\Lambda,z)(v)=\sum_{\gamma\in \Gamma; \langle \phi_j,\gamma+z\rangle \neq 0}
\frac{e^{2i\pi \langle v,\gamma+z\rangle }}{\prod_{j=1}^N 2i\pi \langle \phi_j,\gamma+z\rangle }.$$
Clearly, $$\Eis(\Phi,\Lambda,z)(v)=e^{2i\pi \langle v,z\rangle }  {\mathcal B}(\tilde \Phi,\Lambda)(v)$$
for $\tilde \Phi=[[\phi_1,\langle \phi_1,z\rangle ],\cdots [\phi_N,\langle \phi_N,z\rangle ]]$.
If $z$ is regular, that is $\langle \phi,\gamma\rangle +z\neq 0$
for all $\phi \in \Phi$, then $\Eis(\Phi,\Lambda,z)(v)$ defines a distribution of $v$ with coefficients
meromorphic functions on $T_\C= U_\C/\Gamma$ which is studied in \cite{brver2}.

\begin{example}\label{eg:beraffine}
Let $\Lambda=\Z \omega$, and let $\tilde \Phi_k=[[\omega,z],[\omega,z],\ldots,
[\omega,z]]$ where   $[\omega,z]$ is repeated $k$ times.
If $z$ is integral, we simply have
$$\mathcal{B}(\tilde \Phi_k,\Lambda)(t\omega)=e^{-2i\pi zt}\mathcal{B}(\Phi_k,\Lambda)(t\omega).$$

If $z$ is not integral and $k=1$, then, using Lemma 16 of \cite{brver2},
\begin{equation}\label{jump1}
\mathcal{B}(\tilde \Phi_1,\Lambda)(t\omega)=\sum_{n \in \Z} \frac{e^{2i\pi nt}}{
2i\pi (n+z)}dt=\frac{e^{([t]-t)2i\pi z}}{1-e^{-2i\pi z}},
\end{equation}
which is an analytic function of $t$ in each tope.
%We remark that the jump from the region $t>0$ to $t<0$ between both functions  $\mathcal{B}(\tilde \Phi_1,\Lambda)(t\omega)$
%at any  integer is the function $e^{-2i\pi z t}$.

If $k>1$, $z$ not integral, and $0<t<1$, we use the residue theorem for the integral
$$\int_{|u|=R}\frac{e^{- tu}}{(2i \pi z-u)^k (1-e^{- u})}du$$
which tends to $0$ when $R$ tends to infinity.  Then,
\[\begin{array}{ll}
\mathcal{B}(\tilde \Phi_k,\Lambda)(t\omega)&=-\Res_{u=2i\pi z}\frac{e^{- tu}}{(2i \pi z-u)^k (1-e^{- u})}du\\
&=e^{-2i\pi zt}\Res_{u=0}\frac{e^{tu}}{u^k(1-e^{-2i\pi z+u})}du.
\end{array}\]
Thus, we see that $\mathcal{B}(\tilde \Phi_k,\Lambda)(t\omega)$ is a product of an exponential function of $t$ and a polynomial in $t$.
In particular, in the interval $0<t<1$, it is an analytic function of $t$.
For example, for $k=2$ and $0<t<1$, we get
$$\mathcal{B}(\tilde \Phi_2,\Lambda)(t\omega)=\frac{e^{-2i\pi zt}}{1-e^{-2i\pi z}}
\left(t+\frac{1}{e^{2i\pi z}-1}\right).$$
\end{example}

\subsection{Recurrence relations}
In the affine case the recurrence relation (\ref{induction}) is slightly modified.

Let $\tilde \phi=[\phi,z]$ be an element of $\tilde \Phi$.
We consider two cases.

$\bullet$ Suppose there exists $\gamma_z\in \Gamma$ such that
\begin{equation}\label{chpar}\langle \phi,\gamma_z\rangle +z=0.\end{equation}
Then, we may express $\gamma \in \Gamma_{\reg}(\tilde \Phi - \{\tilde \phi\})$
satisfying $ \ll \gamma,\phi\rr+z=0$ as $\gamma=\gamma'+\gamma_z$. Clearly, $\langle \gamma',\phi \rangle=0$.

\medskip

We consider the system
$$\tilde \Phi_0=[[\overline\phi_j, z_j+ \langle \phi_j,\gamma_z\rangle ],\;\phi_j\in \Phi-\{\phi\}]$$
in $V_0=V/<\phi>$. The sum
$$\sum_{\gamma' \in \Gamma_{\reg}(\tilde \Phi-\{\tilde \phi\}), \ll
\gamma',\phi\rr=0} \frac{e^{\ll 2i\pi v ,\gamma'\rr}}{\prod_{\phi_j\in \Phi-\{\phi\}} 2i\pi (\ll \phi_j,\gamma' \rr+
\ll \phi_j,\gamma_z \rr +z_j)}$$
is constant in the direction of $\phi$ and identifies with
$\mathcal{B}(\tilde \Phi_0,\Lambda)(\overline v)$.
Hence, we get the following recurrence relation.
\begin{equation}\label{induction2}
(\partial_\phi +2i\pi z) \mathcal{B}(\tilde\Phi,\Lambda)(v)
=\mathcal{B}(\tilde\Phi-\{\tilde\phi\},\Lambda)(v)-
e^{2i\pi\langle v,\gamma_z\rangle }\mathcal{B}(\tilde \Phi_0,\Lambda_0)(\overline v).
\end{equation}

\medskip

$\bullet$ If there does not exist $\gamma_z$ satisfying
the relation (\ref{chpar}), then $\Gamma_{\reg}(\tilde \Phi - \{\tilde \phi\})=\Gamma_{\reg}(\tilde \Phi)$, and the
equation (\ref{induction2}) reduces to $$(\partial_\phi +2i\pi z) \mathcal{B}(\tilde\Phi,\Lambda)(v)
=\mathcal{B}(\tilde\Phi-\{\tilde\phi\},\Lambda)(v).$$

\subsection{Piecewise exponential polynomial behavior}

We consider $\tilde \phi=[\phi,z]\in \tilde \Phi$ with $\phi \neq 0$. Consider the complex hyperplane
$H_{\tilde{\phi}}:=\{u\in U_{\mathbb{C}}: \ll u,\phi\rr+z= 0\}$.
Consider the set $$\tilde \CH=\CH(\tilde \Phi)=\{H_{\tilde{\phi}},\; \tilde \phi \in \tilde \Phi \}$$
of hyperplanes in $U_\C$.
We denote by $\mathcal{R}_{\tilde{\CH}}$ the ring of rational functions on $U_\C$ with poles along $\tilde \CH$.
That is, if $S(V_\C)$ denotes the symmetric algebra of $V_\C$, identified with the ring of polynomial functions on $U_\C$, then
$\mathcal{R}_{\tilde{\CH}}$ is the ring $S(V_\C)$ of polynomial functions on $U_\C$ together with inverses of forms
$\ll \phi, \cdot\rr+z$ for $[\phi,z] \in \tilde\Phi$.

For $g \in \mathcal{R}_{\tilde{\CH}}$
we define the distribution $\mathcal{B}(\tilde{\CH},\Lambda,g)$ on $V$ by
$$\mathcal{B}(\tilde{\CH},\Lambda,g)=\sum_{\gamma \in \Gamma_{\reg}(\tilde \CH)} g(\gamma)e^{2i\pi \ll v,\gamma \rr} d_{\Lambda}v,$$
where $\Gamma_{\reg}(\tilde \CH)=\Gamma_{\reg}(\tilde \Phi)$, as regularity does not depend on the multiplicity of an element in $\tilde \Phi$.

We fix $\tilde \phi \in \tilde \Phi$, and define $\tilde \CH':=\tilde \CH \setminus H_{\tilde \phi}$.
For $g \in \mathcal{R}_{\tilde{\CH'}}$, we compare $\mathcal{B}(\tilde{\CH},\Lambda,g)$ and
$\mathcal{B}(\tilde{\CH'},\Lambda,g)$.

Similar to the nonaffine case, for a fixed $\tilde \phi \in \tilde \Phi$, we define $\tilde \CH':=\tilde \CH \setminus H_{\tilde \phi}$
and $\tilde \CH_0$ to be the collection of affine hyperplanes $H\cap H_{\tilde{\phi}}$ for those $H \in \tilde \CH$ not parallel to  $H_{\tilde \phi}$,
that is, for $H \in \tilde \CH$ associated to $[\phi_j,z_j]\in \tilde \Phi$ with $\phi_j\neq \phi$.
The collection $\tilde \CH_0$ is a collection of affine hyperplanes in the affine space
$H_{\tilde{\phi}}$.

\medskip

We consider two cases:

$\bullet$
There exists $\gamma_z\in \Gamma$ lying in $H_{\tilde \phi}$.  Thus $\ll\phi,\gamma_z\rr+z=0$.
Let $H_0$ be the real hyperplane with equation $\phi=0$.
If $K\in  \tilde \CH_0$, then $K-\{\gamma_z\}$ is a complex hyperplane in $(H_0)_\C$.
Let $\tilde \CH_0^z$ be the collection of hyperplanes $K-\{\gamma_z\}$ with $K\in \tilde \CH_0$.
Then, for  $g \in R_{\tilde \CH'}$, we define $g_0(u):=g(u+\gamma_z)$ lying in $\mathcal{R}_{\tilde \CH_0^z}$.
Let $V_0=V/\R\phi$, and $\Lambda_0$ the image of $\Lambda$ in $V_0$.

It immediately follows from the set theoretic partition in the proof of Proposition \ref{prop:recug} that:

\begin{lemma}\label{recgen}
If $g \in R_{\tilde \CH'}$, then
$$\mathcal{B}(\tilde \CH,\Lambda,g)=\mathcal{B}(\tilde \CH',\Lambda,g)-
e^{2i\pi\langle v,\gamma_z\rangle }p^*
\mathcal{B}(\tilde \CH_0^z,\Lambda_0,g_0).$$
\end{lemma}

$\bullet$
In the case that there does not exist any $\gamma_z\in \Gamma$ lying in $H_{\tilde{\phi}}$ and satisfying Equation (\ref{chpar}), we have
$$\mathcal{B}(\tilde \CH,\Lambda,g)=\mathcal{B}(\tilde \CH',\Lambda,g).$$

\bigskip
For a fixed $\tilde \Phi$ we will denote the list of vectors $\phi$ coming from the first component of the pairs in $\tilde\Phi$ by $\Phi$.
Suppose that the vectors in $\Phi$ associated to $\tilde\Phi$ span $V$.  Let $\mathcal{G}_{\tilde{\CH}}$ denote the
subspace of $\mathcal{R}_{\tilde{\CH}}$ generated by functions of the form $$\tilde \theta(L)(x)=\frac{1}{\prod_{\alpha \in L}\ll \alpha,x\rr+z_{\alpha}}$$ where
$L$ is a list of vectors coming from $\Phi$ generating $V$.

We call a function that is a sum of products of exponential functions and polynomial functions an {\it exponential polynomial}.

We will say that a locally $L^1$ function $f$ is \textit{piecewise exponential polynomial}, if there exists a decomposition of $V$ in a union of
polyhedral pieces $C_i$ such that  the restriction of $f$ to $C_i$
is given by a  exponential polynomial formula.  We then say that the distribution
$f(v)d_{\Lambda}v$ is piecewise exponential polynomial.

\begin{proposition}
If $g \in \mathcal{G}_{\tilde{\CH}}$, then $\mathcal{B}(\tilde \CH,\Lambda,g)$ is a piecewise  exponential polynomial distribution.
\end{proposition}

\begin{proof}
We use the same line of argument as in the proof Proposition \ref{compute}.  As before, we
scale the denominator of $g=\tilde \theta(L)$ such that all $\alpha \in L$ lie in
the lattice $\Lambda$.  In the case that $L$ has independent elements, $\mathcal{B}(\tilde \CH,\Lambda,g)$ can be written as
a product of exponential polynomial functions $\mathcal{B}(\tilde \Phi_k,\Lambda)$, whose expression changes whether
the (scaled) $z$ are integral or not.  The expression for both cases is given explicitly in example \ref{eg:beraffine}
and they are piecewise exponential polynomials.  We then use the averaging formula (\ref{E:aver3}).

In order to reduce the general case to the case of independent vectors we use an analogue of Lemma \ref{ind}, and in
the case that same $\alpha$ with distinct $z$ appears in $\tilde
\Phi$, we use the relation
$$\frac{1}{(\alpha+z_1)(\alpha+z_2)}=\frac{1}{z_1-z_2}\frac{1}{(\alpha+z_1)}+\frac{1}{z_2-z_1}\frac{1}{(\alpha+z_2)}.$$
We then get the claimed property of $\mathcal{B}(\tilde \CH,\Lambda,g)$ by induction using Lemma \ref{recgen}.
\end{proof}

The above proposition for $\tilde
\Phi=[[\phi_1,z_1],\ldots, [\phi_N,z_N]]$ and $$g(x)=\frac{1}{\prod_{j=1}^N 2i\pi (\langle \phi_j,x\rangle +z_j)}$$ gives:

\begin{corollary}
If $\Phi$ associated to $\tilde \Phi$ generates $V$, then $\mathcal{B}(\tilde \Phi,\Lambda)(v)$ is an exponential polynomial
function of $v$ on a tope of $\CT(\Phi,\Lambda)$.
\end{corollary}

\begin{remark} Using the same proof as above, we see that
$\Eis(\Phi,\Lambda,z)(v)=e^{2i\pi \langle v,z\rangle } {\mathcal B}(\tilde \Phi,\Lambda)(v)$
is an exponential polynomial function of $v$ on each tope $\tau$ in $\CT(\Phi,\Lambda)$.
Furthermore, when $z$ is regular, the recurrence relation simplifies to
\begin{equation}\label{induction3}
\partial_\phi  \Eis( \Phi,\Lambda,z)=\Eis(\Phi\setminus\{ \phi\},\Lambda,z).
\end{equation}

The system of relations in (\ref{induction3}) are the relations of Dahmen-Miccelli \cite{dm}.
In particular, on each tope $\tau$, we obtain that
$\Eis(\Phi,\Lambda,z)(v)=\sum  K_i(v) F_i(z)$ where $K_i(v)$ are Dahmen-Micchelli polynomials
and $F_i(z)$ meromorphic functions of $z$.
\end{remark}

\subsection{Wall crossing}

Given a tope $\tau$ in $\CT(\Phi,\Lambda)$, we denote by
$\Ber(\tilde \Phi,\Lambda,\tau)$ the polynomial exponential  function  on $V$ such
that the restriction of $\mathcal{B}(\tilde\Phi,\Lambda)$ to  $\tau$
coincides with the restriction of $\Ber(\tilde\Phi,\Lambda,\tau)(v)
dv$ on $\tau$.

Let $\langle H(\alpha,z)|f\rangle =\int_{t>0} f(t \alpha) e^{-2i\pi tz} dt$.

Given a wall $W$,  assume that we have renumber $\tilde \Phi$
so that $\tilde\Phi=[[\phi_1,z_1],\ldots, [\phi_p,z_p],[\phi_{p+1},z_{p+1}],\ldots, [\phi_{p+q},z_{p+q}]]$
where the first $p$ elements $\phi_k$ belongs to $W$  and the last $q$ elements $\phi_{p+j}$ do not belong to $W$.
Then, we define the lists
$$\tilde \Phi\cap W:=[[\phi_1,z_1],\ldots, [\phi_p,z_p]]$$
and
$$\tilde \Phi\setminus W:=[[\phi_{p+1},z_{p+1}],\ldots, [\phi_{p+q},z_{p+q}]].$$

Let $E$ be an equation for the wall $W$.  We define
$$T(\tilde \Phi\setminus W,E):=\displaystyle
\prod_{\langle \phi_i,E\rangle <0}-H(-\phi_i,-z_i)*\prod_{\langle \phi_i,E\rangle > 0}
H(\phi_i,z_i).$$

\medskip

We remark that due to the periodicity property of $\mathcal{B}(\tilde \Phi,\Lambda)$ it suffices to consider jumps over an hyperplane
$W$ passing through the origin.  We have, similar to Theorem \ref{T:jumpber},

\begin{theorem}\label{T:jumpaffine}
Let $\tau_1$ and $\tau_2$ be two adjacent topes of $\CT(\Phi,
\Lambda)$ separated by the hyperplane $W$, with equation $E$. Assume that $\ll v,E\rr>0$ for $v\in \tau_1$. Denote by $\tau_{12}$ the
tope in $\CT(\Phi \cap W,\Lambda\cap W)$ containing
$\overline{\tau_1} \cap \overline{\tau_2}$ in its closure. Let
$\Ber(\tilde \Phi\cap W,\Lambda \cap W,\tau_{12})dh$ be the analytic density
on $W$ determined by $\tau_{12}$. Then, $$(\Ber(\tilde
\Phi,\Lambda,\tau_1)-\Ber(\tilde \Phi,\Lambda,\tau_2))dv$$
$$= \Ber(\tilde
\Phi\cap W,\tau_{12})*T(\tilde \Phi \setminus W,E)-\Ber(\tilde
\Phi\cap W,\tau_{12})*T(\tilde \Phi \setminus W,-E).$$
\end{theorem}

\begin{proof}
The proof follows the same line of argument as in the proof of Theorem \ref{T:jumpber}.  For the first inductive step,
 we are reduced by the same argument as in Theorem   \ref{T:jumpber} to a product situation of $W$ with the line $\R \phi$.
 Then we compute explicitly using Formula (\ref{jump1}).
\end{proof}

\begin{example} Recall the data of Example \ref{eg:beraffine}.  Let $\tau_1$ and $\tau_2$ be
two adjacent topes
defined by inequalities $0<t<1$ and $-1<t<0$ respectively. By
Theorem \ref{T:jumpaffine},
$$\Ber(\tilde \Phi_1,\Lambda,\tau_1)(t)-\Ber(\tilde \Phi_1,\Lambda,\tau_2)(t)=e^{-2i\pi z t},$$ which is also
seen from the explicit expression of $\mathcal{B}(\tilde \Phi_1,\Lambda)(t\omega)$ in
Example \ref{eg:beraffine}.
\end{example}

\section{A decomposition formula}\label{section:decomp}
Let $\Lambda$ and $\Phi$ be as before. We do not necessarily assume that $\Phi$ generates $V$.

In this section, we express $\mathcal{B}(\Phi,\Lambda)$ as a sum of distributions $\mathcal A(\Phi,\Lambda,\a,\beta)$ associated to affine admissible subspaces $\a$ and a generic vector $\beta$ in $V$.

Let us start the construction of   the distribution  $\mathcal A(\Phi,\Lambda,\a,\beta)$.
\medskip

Let $\s$ be a $\Phi$-admissible subspace of $V$. Then $\Phi\cap \s$ generate $\s$, and $\Lambda\cap
\s$ is a lattice in $\s$.  Let $\tau$ be a tope in $\CT(\Phi\cap\s,\Lambda\cap\s)$. We can then consider the distribution
$B(\tau)(s):=\Ber(\Phi\cap \s, \Lambda\cap
\s,\tau)(s)ds$. It is a polynomial density on $\s$.
We still denote by $B({\tau})$ this distribution considered as a
distribution on $V$:
$$\langle B(\tau),test\rangle =\int_{\s} test(s) B(\tau)(s) ds.$$

Let $\lambda\in \Lambda$. Then $\a:=\lambda+\s$ is an affine
$\Phi$-admissible subspace of $V$.
We say that $\a$ is of direction $\s$.
By definition,  a tope  $\tau$ of $\a$ is such that
$\tau-\lambda$ is a tope in $\s$.  We define $B(\Phi\cap\s,\tau)$ as a distribution supported on $\a$
by the formula
$$\langle B(\Phi\cap\s,\tau),test\rangle =\int_{\s} test(s+\lambda) B(\tau-\lambda)(s).$$
We remark that the definition of $B(\Phi\cap\s,\tau)$ above depends only
on $\tau$ and not on the choice of $\lambda$. Indeed, for
another $\lambda'\in \Lambda$ such that $\a=\lambda'+\s$, $\lambda'$ is necessarily
of the form $\lambda'= \lambda+\lambda_0$ for some $\lambda_0 \in \Lambda\cap
\s$. Then,
$$\int_{\s} test(s+\lambda+\lambda_0) B(\tau-\lambda-\lambda_0)(s)
= \int_{\s} test(s+\lambda)
B(\tau-\lambda-\lambda_0)(s-\lambda_0).$$ Using relation
(\ref{periodic}), we have
$B(\tau-\lambda-\lambda_0)(s-\lambda_0)=B(\tau-\lambda)(s)$, hence the independence of the expression.

\bigskip

For a $\Phi$-admissible subspace $\s$, consider an element $u\in U$ vanishing on $\s$ and polarizing
for $\Phi\setminus \s$.  Then,  the multispline distribution $T(\Phi\setminus \s,u)$  is well defined.

\begin{definition}
Let $\a$ be a $\Phi$-admissible affine subspace of $V$ of direction $\s$. Let
$\tau$ be a tope in $\a$, and let $u\in U$ be a vector vanishing on $\s$ and polarizing
for $\Phi\setminus \s$. Then, we define
$$\A(\Phi,\Lambda,\a,\tau,u):= B(\Phi\cap \s,\tau)* T(\Phi\setminus
\s,u).$$
\end{definition}

The distribution $\A(\Phi,\Lambda,\a,\tau,u)$ is supported on $\a+u_{\geq 0}$. It is polynomial in the direction $\s$.

\begin{remark}
Choose a direct sum decomposition  $V=\s\oplus \mathfrak r$ and express $v\in V$ as $v=s+r$ for
$s\in \s$ and $r\in \mathfrak r$.  If  $\Phi$ is equal to $\Phi\cap \s\oplus \Phi\cap \mathfrak r$,
then the function $\A(\Phi,\Lambda,\a,\tau,u)$ is, in product coordinates $(s,r)$, the product of
$B(\Phi\cap \s,\tau)(s)ds$ with $T(\Phi\setminus\s,u)(r)$.
In general it is still possible to express
$\A(\Phi,\Lambda,\a,\tau,u)(s,r)$  as a linear combination of product of multispline functions on $\mathfrak r$ and polynomials on $\s$.
\end{remark}

\medskip

Our main theorem is that $\mathcal{B}(\Phi,\Lambda)$ can be decomposed as a
sum of distributions $\A(\Phi,\Lambda,\a,\tau,u)$ over all
$\Phi$-admissible affine subspaces $\a$ for conveniently chosen
$\tau$ and $u$.  Thus we think of the distributions   $\A(\Phi,\Lambda,\a,\tau,u)$ as the basic building
blocks of the theory.

\bigskip
Choose a scalar product  $\langle,\rangle$ on $V$.
If $W$ is a subspace of $V$, or a quotient space of $V$, then $W$ inherits a scalar product.

\medskip

Let $\beta\in V$, and let $\a$ be a $\Phi$ admissible affine subspace of direction $\s$. We
can then write $\beta=\beta_0-\beta_1$ where $\beta_0\in \a$ and
$\beta_1\in \s^{\perp}$. The point $\beta_0$ is the orthogonal
projection of $\beta$ on $\a$. Assume $\beta$ generic so that

$\bullet$ the point  $\beta_0$  lies in a  tope $\tau(\beta_0)$ of
$\a$.

$\bullet$ the element $\beta_1$  is polarizing  for $\Phi\setminus \s$:
  $\langle \phi,\beta_1\rangle\neq 0$ for all $\phi\in \Phi$ and not in $\s$.

We can then define
$$\A(\Phi,\Lambda,\a,\beta):= \A(\Phi,\Lambda,\a,\tau(\beta_0),\beta_1).$$

\begin{theorem}\label{decomposition}
Choose $\beta\in V$  generic. Then, we have
\begin{equation}\label{E:decomp}
\mathcal{B}(\Phi,\Lambda)=\sum_{\a} \A(\Phi,\Lambda,\a,\beta).
\end{equation}
Here the sum is over all admissible affine subspaces $\a$.

\end{theorem}

            The sum above is infinite. But
remark that, given a vector $v \in V$ by the definition of $\A(\Phi,\Lambda,\a,\beta)$,
there exists only finitely many $\Phi$-admissible affine spaces $\a$ such that $\A(\Phi,\Lambda,\a,\beta)$
gives a non zero contribution at the element $v\in V$, therefore
the above sum is well defined.

For example, if $\s=0$, then the affine spaces $\a$ of direction $\s$ are reduced to the points
$\lambda$ in $\Lambda$ and $$\A(\Phi,\Lambda,\{\lambda\},\beta)=\delta_\lambda*T(\Phi,\lambda-\beta).$$
We see that $\A(\Phi,\Lambda,\{\lambda\},\beta)$  is supported in an affine space $\lambda+\xi$ with
$\langle \xi,\lambda-\beta\rangle >0$. Thus the points $v$ in the support
satisfy  $\|v\|^2\geq \|\lambda\|^2-\|\beta\|^2$.
In particular the sum of the distributions
$$\sum_{\lambda\in \Lambda}
\A(\Phi,\Lambda,\{\lambda\},\beta)$$ is well defined.
Similar estimates hold for any admissible subspace $\s$, when  considering the sum over all affine spaces of direction $\s$ .

\begin{remark}
If $\Phi$ generates $V$, then $V$ is admissible, and the term corresponding to $V$ is the polynomial density
$\Ber(\Phi, \Lambda,\tau(\beta))$, with $\tau(\beta)$ the tope containing $\beta$.
 The other distributions $\A(\Phi,\Lambda,\a,\beta)$ with $\a\neq V$ are piecewise polynomial densities with support not intersecting $\tau$.
 \end{remark}

Theorem \ref{decomposition} has the following meaning: although the distribution
$\mathcal{B}(\Phi,\Lambda)$ is very complicated, it is however
obtained by superposing simpler functions which are products of
polynomials and multisplines.

Before giving the proof of this theorem we demonstrate the decomposition in
various examples and state a recurrence relation.

\begin{example}
Let $\Phi=\emptyset$.  Then, by definition,
$\mathcal{B}(\Phi,\Lambda)(v)=\sum_{\gamma\in
\Lambda^*}e^{2i\pi\langle \gamma,v\rangle }$. We would like to decompose this sum
as in equation (\ref{E:decomp}). Observe that in this case $\a$
consists of points of $\Lambda$ and any $\beta$ in $V$ is generic.
Using $T(\emptyset,\beta-\lambda)=\delta_0$, the decomposition in
Theorem \ref{decomposition} gives
$$\sum_{\gamma\in \Lambda^*}e^{2i\pi\langle \gamma,v\rangle }=\sum_{\lambda\in
\Lambda}\delta_{\lambda}(v),$$ which is the Poisson formula.
\end{example}

\begin{example}(one dimensional case)
Let $\Lambda=\Z \omega$, and let $\Phi_k=[\omega,\omega,\ldots,
\omega]$, where  $\omega$ is repeated $k$ times.  Then $\s=0$ or
$\s=V=\R \omega$, correspondingly $\a$ are reduced to points
$\{\lambda\}$ in $\Lambda$ or $\a=V$. Choose any $\beta=r\omega \in
V$ with $0<r<1$, it is generic. The polynomial
$\Ber(\Phi,\Lambda,\tau(\beta))(t\omega)$ which coincide with
$\mathcal{B}(\Phi_k,\Lambda)(t)$ on $0<t<1$ is
$-\frac{1}{k!}B(k,t)$, where $B(k,t)$ is the Bernoulli polynomial.
Then, with the notation of Example \ref{eq:spline}, Theorem \ref{decomposition} gives,

$$\begin{array}{ll}
\mathcal{B}(\Phi_k,\Lambda)(t\omega)=&-\frac{1}{k!}B(k,t)dt+\displaystyle
\sum_{n\in\Z_{>0}} \delta_{n\omega}*T(\Phi_k,\omega^*) \\
&+\displaystyle
\sum_{n\in\Z_{\leq 0}} \delta_{n\omega}*T(\Phi_k,-\omega^*).
\end{array}$$

In Figure \ref{spline2} we depict the decomposition of $\mathcal{B}(\Phi_2,\Lambda)(t\omega)$. In part $(a)$ we draw the graph of the periodic polynomial
$-\frac{1}{2}B(2,t-[t])$, the red graph in part $(b)$ is the graph of the polynomial $-\frac{1}{2}B(2,t)$ and lines in black correspond to contribution
of splines.

\end{example}

\begin{figure}
\begin{center}
\includegraphics[width=40mm]{spline2.mps}\hspace{2cm}
 \includegraphics[width=40mm]{spline1.mps}\\

  (a)\hspace{5.5cm}(b)
 \caption{The decomposition of $\mathcal{B}(\Phi_2,\Lambda)(t\omega)$.}\label{spline2}
 \end{center}
\end{figure}

Let us now study the recurrence relations that the distributions $\mathcal A(\Phi,\Lambda,\a,\beta)$  satisfy. It will be convenient to
define  $\mathcal A(\Phi,\Lambda,\a,\beta)$ for any affine subspace $\a$, by declaring it  to be equal to zero if $\a$ is not admissible.
If $\a=\lambda+\s$ where $\phi\in \s$, we denote by $\a/<\phi>$ the image of the rational space
$\a$ in $V_0=V/<\phi>$.

\begin{lemma}\label{recurrence}
Let $\phi\in \Phi$ and $\beta\in V$.  We still denote by $\beta$
the projection of $\beta$ on $V/<\phi>$.

$(i)$ If $\phi\notin \s$, then
$$\partial_\phi \A(\Phi,\Lambda,\a,\beta)= \A(\Phi\setminus
\{\phi\},\Lambda,\a,\beta).$$

$(ii)$ If $\phi\in \s$, then
$$\partial_\phi \A(\Phi,\Lambda,\a,\beta)= \A(\Phi\setminus
\{\phi\},\Lambda,\a,\beta)-\A(\Phi/<\phi>,\Lambda/<\phi>,\a/<\phi>,\beta).$$
\end{lemma}

In $(ii)$,  $\A(\Phi\setminus
\{\phi\},\Lambda,\a,\beta)$ is zero when $\Phi\cap \s \setminus \{\phi\}$ does not generate $\s$.

\begin{proof}
Part $(i)$ follows from the relation
$$\partial_\phi T(\Phi\setminus\s,\beta_1)= T((\Phi\setminus \{\phi\})\setminus \s,\beta_1).$$

For part $(ii)$ we use equation (\ref{induction}) on $\tau$, which gives
$$\partial_\phi \Ber(\Phi \cap \s,\Lambda\cap\s ,\tau)=
\Ber((\Phi\cap \s)\setminus \{\phi\},\Lambda\cap \s,\tau)-\Ber(\Phi\cap
\s/<\phi>,\Lambda\cap \s/<\phi>,\tau).$$

Now suppose $q$ is a polynomial function on $\s$ constant in the
direction of $\phi$. Let $X:=[v_1,v_2,\ldots, v_N]$ be a sequence of nonzero vectors in $V\setminus \s$
 generating a pointed cone. We denote the projection of
$v_i$ to $V/<\phi>$ by $\bar v_i$. Then,
$$ \begin{array}{ll} q*T(X)(v)&=\displaystyle \int_{0}^{\infty}\cdots \int_{ 0}^{\infty}
q(v-\sum t_i v_i)dt_1\cdots dt_N\\
&=\displaystyle
\int_{0}^{\infty}\cdots \int_{0}^{\infty} q(\overline v-\sum t_i\bar{v_i}) dt_1\cdots dt_N\\
&= \displaystyle q*T(X/<\phi>)(\overline v).
\end{array}$$
Putting $q=\Ber(\Phi\cap \s/<\phi>,\Lambda\cap \s/<\phi>,\tau)$ and
$X=\Phi\setminus \s$, we get part $(ii)$.
\end{proof}

\begin{proof}
We now prove Theorem \ref{decomposition} by induction on the number of elements in
$\Phi$. We assume that the theorem is true for any sublist  of $\Phi$.
Denote by ${\rm Req}(\Phi)$ the right hand side and by ${\rm Leq}(\Phi)$ the left hand side of equation (\ref{E:decomp}).

Let $\phi\in \Phi$. Let $\Phi'=\Phi\setminus \{\phi\}$. We recall
equation (\ref{induction}),  $$\partial_\phi
\mathcal{B}(\Phi,\Lambda)=\mathcal{B}(\Phi',\Lambda)-\mathcal{B}(\Phi/<\phi>,\Lambda/<\phi>).$$

Let $\CR_{\rm aff}(\Phi)$ be the collection of all $\Phi$-admissible affine  subspaces of $V$.
Let  $\CR^0$ be the subset consisting of the elements $\a$ whose direction $\s$ contains $\phi$.

We now differentiate ${\rm Req}(\Phi)$ with respect to $\phi$.  Using relations given in part $(i)$ and
$(ii)$ of Lemma \ref{recurrence}, we get
$$
\partial_\phi {\rm Req}(\Phi)= \sum_{\a\in \CR_{\rm aff}(\Phi)} \A (\Phi',\Lambda, \a,\beta)
-\sum_{\a\in \CR^{0}} \A (\Phi/<\phi>,\Lambda/<\phi>,\a/<\phi>,\beta).
$$
We observe that the collection of $\a/<\phi>$  with $\a\in  \CR^0$ parametrizes all  affine spaces  admissible
for $\Phi/<\phi>$. The collection   $\CR_{\rm aff}(\Phi)$ may be larger than
$\CR_{\rm aff}(\Phi')$, but, if $\a$ is not in $\CR_{\rm aff}(\Phi')$, then the contribution
$\A (\Phi',\Lambda, \a,\beta)$ is equal to $0$.

Hence, we obtain by induction that $\partial_\phi({\rm Leq}(\Phi)-{\rm Req}(\Phi))=0$ for any $\phi$.

Thus
${\rm Leq}(\Phi)-{\rm Req}(\Phi)$ is constant. But by construction
${\rm Leq}(\Phi)-{\rm Req}(\Phi)$ is equal to zero on $\tau$, therefore the constant is zero.
\end{proof}

\begin{example}
We will give a decomposition formula for the system in Example \ref{bernoulli2}.
We recall the data: $\Lambda=\Z e_1\oplus \Z e_2$, $\Phi=[e_1,e_2,e_1+e_2]$.

We will compute $\mathcal{B}(\Phi,\Lambda)(v)$ using the
decomposition formula for $v$ in the tope defined by the
inequalities $0<v_2<1$, $1<v_1<2$ and $v_1-v_2>1$ (see figure
\ref{decompbeta}).

\medskip
We aim to demonstrate the dependence of the summands in the
decomposition formula in Theorem \ref{decomposition} to the chosen generic point in
a tope, though the value of $\mathcal{B}(\Phi,\Lambda)(v)$ is
clearly independent of this choice. We will thus decompose
$\mathcal{B}(\Phi,\Lambda)(v)$ in two different ways, for two
different choices of generic points lying in the \textit{same} tope.

\begin{figure}
\begin{center}
  \includegraphics[width=70mm]{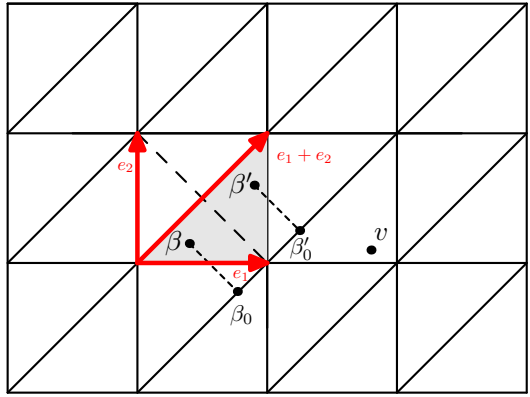}\\
\caption{Decomposition for various generic points}\label{decompbeta}
\end{center}
\end{figure}

For the first (resp.~second) computation we choose a generic $\beta=b_1e_1+b_2e_2$ (resp.~$\beta'$)
in the tope defined by $0<b_1<1$, $0<b_2<1$ and $b_1>b_2$, and further satisfying $b_1+b_2<1$ (resp.~$b_1+b_2>1$).

Figure \ref{decompbeta} depicts two such choice of generic elements.
We denote the projection of $\beta$ and $\beta'$ to
$-e_2\oplus \R(e_1+e_2)$ by $\beta_0$ and $\beta_0'$ respectively.
Since these projections
lie in different topes for the reduced system $(\Phi\cap \s,\Lambda\cap \s)$ their corresponding contribution to the sum
in Theorem \ref{decomposition} will be different.  In fact, choosing $\beta$ as the generic point will enforce a nonzero
contribution of the lattice point $(1,0)$ in evaluating the distribution at a point $v$ as depicted in Figure \ref{decompbeta}.

\medskip

\noindent \textit{Computation with generic point $\beta$}:
$$\begin{array}{ll}
\mathcal{B}(\Phi,\Lambda)(v)=&\A(\Phi,\Lambda,V,\beta)(v)+\A(\Phi,\Lambda,\R e_2+e_1,\beta)(v)\\
&+\A(\Phi,\Lambda,\R(e_1+e_2)-e_2,\beta)(v)+\A(\Phi,\Lambda,\{(1,0)\},\beta)(v).
\end{array}$$

We now compute each summand using the formula in Lemma \ref{defPol}.
$$\begin{array}{ll}
\A(\Phi,\Lambda,\R e_2+e_1,\beta)(v)&=\Ber(e_2,\tau(\beta_0))*T(\{e_1,e_1+e_2\},\beta_1)(v)\\
&=\delta_{(1,0)}*\Res_{z=0}\left(\left((1/2-\partial{x_2})\cdot\frac{e^{\langle v, x+ze^1\rangle }}{(x_1+z)(x_1+x_2+z)}\right)_{x=0}\right)\\
%&=\delta_{(1,0)}*\left(\frac{1}{2}v_1(1-2v_2+v_1)\right)\\
&=\frac{1}{2}(v_1-1)(-2v_2+v_1)
\end{array}$$
$$\begin{array}{ll}
\A(\Phi,\Lambda,\R(e_1+e_2)-e_2,\beta)(v)&=\Ber(e_1+e_2,\tau(\beta_0))*T(\{e_1,e_2\},\beta_1)(v)\\
&=\delta_{(0,-1)}*-\Res_{z=0}\left(\left((1/2-\partial{x_1})\cdot\frac{e^{\langle v, x+z(-e^1+e^2)\rangle }}{(x_1-x_2-z)(x_2+z)}\right)_{x=0}\right)\\
%&=\delta_{(0,-1)}*-\left(\frac{1}{2}(1-v_1-v_2)(v_1-v_2)\right)\\
&=-\frac{1}{2}(-v_1-v_2)(v_1-v_2-1)
\end{array}$$
$$\begin{array}{ll}
\A(\Phi,\Lambda,\{(1,0)\},\beta)(v)&=\delta_{(1,0)}*-T(e_1,-e_2,e_1+e_2)=-(v_1-1-v_2)
\end{array}$$
Using the computation in example \ref{bernoulli2} for $\A(\Phi,\Lambda,V,\beta)(v)$, we get
\[\begin{array}{lll}
\mathcal{B}(\Phi,\Lambda)(v)&=& -\frac{1}{6}(v_1-2v_2)(v_1-1+v_2)(2v_1-1-v_2)+\frac{1}{2}(v_1-1)(-2v_2+v_1)\\
& &+\frac{1}{2}(v_1+v_2)(v_1-v_2-1)-(v_1-1-v_2)\\
&=&-\frac{1}{6}(v_1-1-2v_2)(2v_1-3-v_2)(v_1-2+v_2)
\end{array}\]

\medskip

\noindent \textit{Computation with generic point $\beta'$}:
$$\begin{array}{ll}
\mathcal{B}(\Phi,\Lambda)(v)=&\A(\Phi,\Lambda,V,\beta')(v)+\A(\Phi,\Lambda,\R e_2+e_1,\beta')(v)+\\
&\A(\Phi,\Lambda,\R(e_1+e_2)+e_1,\beta')(v).
\end{array}$$
The first two summands in the decomposition above are already computed. The third summand equals:
$$\A(\Phi,\Lambda,\R(e_1+e_2)+e_1,\beta')(v)=-\frac{1}{2}(2-v_1-v_2)(v_1-1-v_2).$$
We then have
\[\begin{array}{lll}
\mathcal{B}(\Phi,\Lambda)(v)&=&-\frac{1}{6}(v_1-2v_2)(v_1-1+v_2)(2v_1-1-v_2)+\frac{1}{2}(v_1-1)(-2v_2+v_1)\\
& &-\frac{1}{2}(2-v_1-v_2)(v_1-1-v_2)\\
&=&-\frac{1}{6}(v_1-1-2v_2)(2v_1-3-v_2)(v_1-2+v_2)
\end{array}\]
as expected.
\end{example}

\medskip

In the affine case, we define
$$\A(\tilde \Phi,\Lambda,\a,\tau,u)= \Ber(\tilde \Phi\cap \s,\tau)* T(\tilde \Phi\setminus
\s,u).$$  We have a decomposition formula analogous to Theorem \ref{decomposition}.

\begin{theorem}\label{decompositionaff}
Choose $\beta\in V$ sufficiently generic. Then we have
\[\mathcal{B}(\tilde \Phi,\Lambda)=\sum_{\a} \A(\tilde \Phi,\Lambda,\a,\beta).\] \end{theorem}

The proof is precisely in the same line of arguments with that of Theorem \ref{decomposition}.


\newpage



\begin{thebibliography}{99}

\bibitem{bover}{\bf Boysal A. and Vergne M.},{\em Paradan's wall crossing formula for partitions functions and
Khovanski-Pukhlikov differential operator.}
Annales de l'Institut Fourier  {\bf 59} (2009), 1715-1752.


\bibitem{brver1}{\bf Brion M. and Vergne M.},{\em Arrangement of hyperplanes I: Rational functions and Jeffrey-Kirwan residue.}
Ann. scient. {\'E}c. Norm. Sup. {\bf 32} (1999),  715-741.

\bibitem{brver2}{\bf Brion M. and Vergne M.},{\em Arrangement of hyperplanes II: The Szenes formula and Eisenstein series.}
Duke Math. J. {\bf 103} (2000), 279--302.

\bibitem{dm}{\bf Dahmen and Micchelli},{\em Translates of  multivariate splines.} Linear Algebra Appl. {\bf 52} (1983), 217--234.


\bibitem{dp}{\bf De Concini C. and Procesi C.},{\em Topics in hyperplane arrangements, polytopes and
box splines.} To appear (available on the personal web page of C. Procesi).

\bibitem{dpv}{\bf De Concini C., Procesi C.  and Vergne M.},
{\em Partition functions and  generalized Dahmen-Micchelli spaces.} arXiv : math/0805.2907. Transformation
Groups {\bf 15} (2010) no 4 , 751-773.


\bibitem{gu}{\bf Guillemin V,  Lerman E. and  Sternberg S.},{\em  Symplectic fibrations and multiplicity diagrams.}
Cambridge University Press 1996.


\bibitem{jef-kir}{\bf Jeffrey, L.C. and Kirwan, F.C.}, {\em Intersection theory on moduli spaces of holomorphic bundles of arbitrary
rank on a Riemann surface.} Ann. of Math. (2) {\bf 148} (1998), no. 1, 109--196.

\bibitem{par}{\bf Paradan P-E.},{\em Wall-crossing formulas in Hamiltonian geometry.}
 arXiv:math/0411306

\bibitem{par2}{\bf Paradan P-E.},{\em The moment map and equivariant cohomology with generalized coefficients.}
Topology {\bf 39} (2001), 401--444.

\bibitem{par3}{\bf Paradan P-E.},{\em Localization of the Riemann-Roch character.} J. Funct. Anal. {\bf 187} (2001), 442--509.


\bibitem{sze1}{\bf Szenes A.},{\em Iterated Residues and Multiple Bernoulli Polynomials.}
International Mathematics Research Notices {\bf 18}, (1998),
937--956.

\bibitem{sze2}{\bf Szenes A.},{\em Residue theorem for rational trigonometric sums and Verlinde's formula.}
Duke Math. J. {\bf 118} (2003), 189--227.

\bibitem{sze-ve}{\bf Szenes A. and Vergne M.},{\em Residue formulae for vector partitions and Euler-MacLaurin sums.} Advances in Applied
Mathematics {\bf 30} (2003), 295--342.

\bibitem{sze-ve2}{\bf Szenes A. and Vergne M.},{\em $[Q,R]=0$ and Kostant partition functions} arXiv:math/1006.4149.

\bibitem{Vergnebox}{\bf Vergne M.}
 {\em  A Remark on the Convolution with Box Splines.} arXiv: math/1003.1574  (to appear in Annals of Mathematics).



\bibitem{wi}{\bf Witten E.},{\em On quantum gauge theories in two dimensions.} Commun.~Math. Phys. {\bf 141} (1991), 153--209.

\bibitem{wi2}{\bf Witten E.},{\em Two dimensional gauge theories revisited.} J.~Geom. Phys. {\bf 9} (1992), 303--368.

\end{thebibliography}
\end{document}